\documentclass[11pt,a4paper]{article}
\usepackage[a4paper,margin=1.8cm]{geometry}
\usepackage{amsmath}
\numberwithin{equation}{section}
\numberwithin{figure}{section}
\usepackage{amsfonts}
\usepackage{amssymb}
\usepackage{longtable}
\usepackage{amsthm}
\usepackage{enumerate}
\PassOptionsToPackage{normalem}{ulem}
\usepackage{ulem}
\usepackage{array}
\usepackage{cases}
\usepackage{graphicx}
\usepackage{url}

\usepackage{natbib}
\usepackage{comment}
\usepackage{color}
\usepackage{dsfont}
\usepackage{float}
\usepackage{subfigure}

\def\d{\mathrm{d}}

\def\as{\mathrm{a.s.}}

\newcommand{\R}{\mathbb{R}}
\newcommand{\E}{\mathbb{E}}

\newcommand{\Px}{\mathbb{P}}
\newcommand{\Pb}{\mathbb{P}}

\newcommand{\F}{\mathcal{F}}
\newcommand{\Fb}{\mathbb{F}}
\newcommand{\Fx}{\mathbb{F}}

\newcommand{\Qb}{\mathbb{Q}}

\newcommand{\G}{\mathcal{G}}
\newcommand{\pa}{\partial}

\definecolor{Magenta}{rgb}{1.0, 0, 1.0}

\definecolor{JungleGreen}{rgb}{0.0, 0.8, 0.10}

\definecolor{linkcolor}{rgb}{0,0,0.502}
\definecolor{urlcolor}{rgb}{1,0,0}

\newtheorem{theorem}{\protect Theorem}[section]
\newtheorem{proposition}[theorem]{Proposition}
\newtheorem{definition}[theorem]{\protect Definition}

\newtheorem{lemma}[theorem]{\protect Lemma}

\newtheorem{corollary}[theorem]{\protect Corollary}

\allowdisplaybreaks[4]

\title{Mean Field Stackelberg Game for Production and Carbon Emission Reduction with State Reflections}

\author{Lijun Bo \thanks{Email: lijunbo@ustc.edu.cn, School of Mathematics and Statistics, Xidian University, Xi'an, 710126, China.}
\and
Zhen Liu \thanks{Email: zhenliu001@cuhk.edu.hk, Department of Statistics and Data Science, The Chinese University of Hong Kong, Shatin, NT, Hong Kong.}
\and
Jingfei Wang \thanks{Email:wjf2104296@mail.ustc.edu.cn, School of Mathematical Sciences, University of Science and Technology of China, Hefei, 230026, China.}
\and
Jiacheng Zhang \thanks{Email: jiachengzhang@cuhk.edu.hk, Department of Statistics and Data Science, The Chinese University of Hong Kong, Shatin, NT, Hong Kong.}
}
\date{ }

\begin{document}
\maketitle

\begin{abstract}                         
Global warming, driven by anthropogenic carbon emissions with transboundary pollution characteristics and irreversible damage, poses an existential threat to human society. This paper develops a novel two-level Stackelberg game with mean field interaction of controls and common noise which integrates hierarchical decision-making under a state-reflected emission dynamics. A central regulator (leader) adjusts product prices to guide $n$ heterogeneous competing regions (followers) while enforcing a hard emission cap via a reflection mechanism that models emergency reductions through a local time process. We establish the existence of an approximate Stackelberg equilibrium and perform sensitivity analysis via Monte Carlo simulations.


\end{abstract}

\section{Introduction}\label{sec:introduction}

Sustained growth in human-caused carbon emissions drives accelerating global warming and climate change, which constitute the top urgent threat to global ecosystems (\citealt{lashof1990relative} and \citealt{matthews2009proportionality}). As an unavoidable byproduct of industrial activity and energy use, carbon emissions are difficult to govern due to two inherent features. First, carbon emissions generate substantial cross-border spillovers: emissions produced within one nation spread across boundaries via atmospheric, hydrological and terrestrial channels, damaging the environments of other countries. Emissions from any single region alter the global climate, with resulting environmental costs shared collectively across the globe. 
Second, the atmosphere's carbon storage capacity is a finite, nonrenewable physical resource, and ecological harm from excess carbon buildup is effectively irreversible over human lifespans. This irreversibility is amplified by the climate system's notable pipeline commitment lag: even if future research reveals greater climate sensitivity than presently anticipated, warming locked in by historical emissions cannot be rapidly undone (\citealt{weitzman2009modeling}).

\quad To characterize those unique properties such as cross-border spillovers and irreversibility of environmental damage, we propose a two-level Stackelberg game under state reflection within the framework of transboundary pollution that can simultaneously capture multi-region strategic interactions and a strict cumulative emission constraint. 
The differential game literature on transboundary pollution is substantial. Early work predominantly adopted deterministic models to determine optimal emissions and allocate abatement costs, e.g., \cite{van1992international,dockner1993international,jorgensen2001incentive,breton2005differential} and see \cite{bertinelli2014carbon,  benchekroun2016impact, de2021equilibrium} for further related deterministic models. Stochastic extensions were pioneered by \cite{yeung2007dynamically} and subsequently applied to industrial pollution \cite{yeung2008cooperative} and emission permit trading \cite{li2014differential, chang2015modeling,li2016dynamic}.

\quad The vast majority of existing transboundary pollution models adopt a single-layer game structure, neglecting the hierarchical decision-making prevalent in real-world environmental governance, where a central authority sets policies implemented by regional entities. Stackelberg game theory explicitly captures this leader–follower relationship. \cite{Long1992} first applied it to transboundary pollution control, demonstrating the leader's first-mover advantage through policy commitment. Subsequent work extended Stackelberg differential games to hybrid energy markets with demand uncertainty \cite{Perera2022}, compared environmental taxes versus tradable permits under finite horizons \cite{Cerqueti2023}, and exploited state-separable structures to obtain closed-form solutions for polluting firm--government conflicts \cite{Halkos2014,Halkos2021}.  
Despite these advances, existing Stackelberg models for transboundary pollution lack a strict cumulative emission hard constraint—crucial for averting irreversible climate damage. 

\quad The introduction of a strict emission cap brings mathematical challenges because it requires modeling a state process confined to a bounded domain.  
Stochastic control and stochastic differential games under reflected diffusions have garnered significant attention in recent years. To name a few, \cite{Borkar05} study a ergodic drift control problem of reflected diffusions. \cite{Bayraktar2019} explore a mean field game (MFG) with state reflections in the context of large symmetric queuing systems under heavy traffic. \cite{BLY2021} handle an optimal tracking portfolio problem with capital injection. \cite{BWY2025} verify the existence of mean field equilibrium (MFE) for a MFG of controls with state reflection. Despite these advancements, the study of MFGs of controls in the presence of reflected state processes remains an open problem.

\quad To fill the aforementioned research gaps, we construct a new Stackelberg game for transboundary pollution, embedding hierarchical decision-making alongside a binding cumulative emission constraint. We study a finite-horizon framework consisting of one central leader authority and $n$ heterogeneous regional followers. The leader first commits to a production price adjustment policy; afterwards, followers engage in a Nash game and independently select production schedules to maximize individual profits. Local production generates region-level emissions that aggregate into a global pollution stock, which characterizes transboundary spillover effects. Followers are mutually coupled via the mean-field production term within total accumulated emissions, and their payoffs are affected by the leader’s policy through price-dependent production costs. By tuning prices to regulate output and pollution and implementing pollution abatement, the leader minimizes overall environmental governance costs while guaranteeing the mean cumulative emission stays below the fixed regulatory cap.

\quad Due to the large number of followers, the mean-field coupling and the cap constraint render exact Stackelberg–Nash solutions intractable. We therefore turn to establish $(\epsilon_1(n),\epsilon_2(n))$-Stackelberg equilibrium to characterize the best estimate of actual mean field behavior (\citealt{moon2018linear}). For this heterogeneous finite-region game, we proceed in two steps. First, given an arbitrary leader's strategy, we solve the followers' mean-field game by deriving the representative follower's best response strategy and then establish a leader-dependent $\epsilon(n)$-Nash equilibrium with $\lim_{n\to\infty}\epsilon(n)=0$. Next, taking $\epsilon(n)$-Nash strategies of followers as given, we solve the leader's optimal control problem and obtain a semi-explicit characterization of her optimal price adjustment and abatement strategies via the optimality system, Skorokhod theory and a probabilistic representation. We lastly show that the established strategies for the leader and the followers constitute an $(\epsilon_1(n),\epsilon_2(n))$-Stackelberg equilibrium with $\lim_{n\to\infty}\epsilon_i(n)=0$ for $i=1,2$. Lastly, we conduct numerical simulations using Monte Carlo methods to analyze the sensitivity of the equilibrium outcomes to model parameters.

\quad In contrast to prior study \cite{BLY2021} that solve the HJB equation via probabilistic approaches, our work resolves multiple non-trivial technical difficulties introduced by our more elaborate model setup. First, the reflected process $Y^{t,x,z}$ defined in \eqref{eq:Y} for our probabilistic representation corresponds to a geometric Brownian motion (GBM) reflected at the upper boundary, while \cite{BLY2021} only consider reflected drifted Brownian motion. Those references rely on the closed-form density for the maximum process of standard Brownian motion; however, such an explicit density is not guaranteed to exist under our model. We resolve this issue using Malliavin calculus: drawing on the arguments in \cite{Coutin2019SPL}, we prove the existence of density for the maximum process of $U^{t,z}$ in \eqref{eq:processU} and additionally confirm its local Lipschitz continuity property. Second, unlike drifted Brownian motion, the drift and diffusion coefficients of our reflected process $Y^{t,x,z}$ inherently depend on the initial state $x$. To eliminate this state dependence, we construct an auxiliary factor process $K^t$ in \eqref{eq:GBMK}, such that the product process $K^tY^{t,x,z}$ reduces to a drifted Brownian motion. Although this transformation mitigates GBM-specific complications, analyzing the regularity of the joint density becomes considerably more complex. We address this challenge in Lemma~\ref{lemma:density} by applying Girsanov’s theorem for measure change, which encapsulates all randomness from $K^t$ into an alternative probability measure. Lastly, Proposition \ref{prop:probPsi} merely confirms our probabilistic solution is a weak solution within a Sobolev space. We therefore apply the generalized It\^o's formula tailored to Sobolev spaces to complete a rigorous proof of the verification theorem.

\quad This paper is structured as follows. In Section \ref{sec:model}, we formulate a two-level Stackelberg game involving one leader and $n$ heterogeneous followers. Section \ref{sec:non-MFG} constructs an approximate Nash equilibrium for the population of $n$ followers. Section \ref{sec:leader} addresses the leader’s optimal control problem subject to penalized reflection constraints and then derives an approximate Stackelberg equilibrium of our two-level Stackelberg game. Numerical sensitivity is examined via experiments undertaken in Section~\ref{sec:number}.

\noindent{\bf Notations.} We list below some notations that will be frequently used throughout the paper:

\noindent $\mathcal{C}_T$ ($\mathcal{C}_T^+$): Set of real-valued (non-negative) continuous functions on $[0,T]$ under the norm $\|f\|_{\infty}:=\sup_{t\in[0,T]}|f(t)|$ for $f\in\mathcal{C}_T$; $\mathcal{P}(E)$ ($\mathcal{P}_p(E)$): Set of probability measures on $E$ (with finite $p$-order moments); $L^2_{\mathbb{F}}(0,T;E)$: Set of $\mathbb{F}$-adapted $E$-valued processes $X=(X(t))_{t\in[0,T]}$ with $\|X\|_2:=\sqrt{\mathbb{E}[\int_{0}^{T}\|X(t)\|_E^2dt]}<\infty$; $\nu(f):=\int f(\boldsymbol{\theta})\nu(d\boldsymbol{\theta})$ for a probability measure $\nu$; $L_{\rm loc}^{\infty}([0,T]\times E)$: Set of locally essentially bounded functions on $[0,T]\times E$, i.e., $f\in L_{\rm loc}^{\infty}([0,T]\times E)$ iff $f \in L^{\infty}(K)$ for any compact subset $K \subset [0,T]\times E$; $\mathcal{W}^{1,2,1;\infty}_{\rm loc}([0,T]\times E)$: Sobolev space consisting of functions $f: [0,T]\times E \to \mathbb{R}$ whose weak derivatives $\partial_t f$, $\partial_x f$, $\partial_{xx} f$ and $\partial_z f$ belong to $L^{\infty}_{\mathrm{loc}}([0,T]\times E)$.

\section{Problem Formulation}\label{sec:model}

In this section, we formulate a Stackelberg game with one central regulator (leader) and \(n\) heterogeneous regions (followers): the followers independently decide their production plans, while the leader  adjusts the product price and implements emission abatement measures subject to an emission cap constraint. Let $T\in(0,\infty)$ be the terminal time horizon. We consider a multinational economic system consisting of $n$ regions (nations or cities), each seeking to regulate production activities while collectively adhering to a average emission cap constraint. To formulate the stochastic optimal production-reduction control problem for this two-level cross-regional economy, we fix a filtered probability space $(\Omega,\F,\mathbb{F},\Px)$ equipped with the filtration $\Fx=(\F_t)_{t\in[0,T]}$ that satisfies the usual conditions and it supports a scalar Brownian motion $W=(W(t))_{t\in[0,T]}$. 

\subsection{Control-Distribution Dependent State Process with Reflection}

Denote by $X^{(n),\boldsymbol{q}}(t)$ the average accumulated carbon emission level across the $n$ regions at time $t\in[0,T]$ under the control strategy $\boldsymbol{q}=(\boldsymbol{q}(t))_{t\in[0,T]}$ with $\boldsymbol{q}(t)=(q_i(t))_{i=1}^n$. As a well-documented feature in environmental economics and control theory (see, e.g., \citealt{van1992international, dockner1993international, yeung2008cooperative}), transboundary pollution exhibits prominent cross-regional diffusion. Specifically, cross-border pollutants such as air and water contaminants generate widespread externalities: emissions originating from one region affect all other jurisdictions, implying that pollution constitutes a common global state rather than independent local outcomes. Under this setup (see also \citealt{BLW2026}), the dynamic evolution of average accumulated carbon emissions is governed by the following control-distribution dependent stochastic differential equation (SDE):
\begin{align}\label{eq:X-q}
d X^{(n),\boldsymbol{q}}(t)=&\left[\frac{1}{n} \sum_{i=1}^{n} \mu q_{i}(t)-\delta X^{(n),\boldsymbol{q}}(t) \right]	d t+\sigma X^{(n),\boldsymbol{q}}(t) d W(t), 
\end{align}
where $X^{(n),\boldsymbol{q}}(0)=x_{0}>0$ is the initial average emission. Here, the emission process of region $i$ is assumed to be proportional to the output process $q_{i}=(q_{i}(t))_{t\in[0,T]}$ in the sense that the emission process of region $i$ is given by $\mu q_{i}(t)$ with $\mu>0$ being the production-emission ratio coefficient. The parameter $\delta\in(0,1)$ is referred to as the carbon absorption rate in nature, and hence $\delta X^{(n),\boldsymbol{q}}(t)dt$ measures the average carbon absorption quantity of these $n$ regions within the period $(t,t+dt]$. The term $\sigma X^{(n),\boldsymbol{q}}(t) d W(t)$ models random measurement errors in the process of carbon emission, which cannot be traced and corrected (\citealt{QuickMarland19}), where $\sigma>0$ is the standard deviation of the measurement error per emission level.

\quad The average cumulative emissions are constrained by a prescribed deterministic process for projected cumulative emissions $Z=(Z(t))_{t\in[0,T]}$, expressed as:
\begin{align}\label{eq:E0-SDE}
d Z(t) = \Phi( Z(t) ) d t , \quad Z(0)=z_{0}\in\R_+:=(0,\infty), 
\end{align}
where $\Phi(\cdot):\mathbb{R}\to\mathbb{R}$ is a Lipschitz function with the Lipschitz coefficient $L>0$ and $\Phi(0)\geq0$, which is determined by the central regulator. The deterministic process $Z$ represents a targeted emission trajectory formulated according to environmental policies or emission reduction goals. 
In fact, the near‑linear ${\rm CO_2}$–warming relationship implies a finite 1.5°C carbon budget (\citealt{matthews2009proportionality,IPCC21,stern2008economics}); exceeding it risks irreversible tipping points, making a stringent absolute cap an indispensable hard constraint (\citealt{weitzman2009modeling, dietz2015endogenous,lenton2019climate}). 
This cap aligns regional emissions with the system‑wide budget and serves as a core condition that bridges regional emission activities and the system-level emission control target.

\quad Building upon \eqref{eq:X-q} and \eqref{eq:E0-SDE}, to ensure compliance with the emission target, all regions adopt an accumulated average abatement process $A^{(n)}=(A^{(n)}(t))_{t\in[0,T]}$, which can represent technologies such as carbon capture and storage (CCS). Consequently, the average accumulated emission process satisfies the following control-distribution dependent SDE with the upper reflection:
\begin{align}\label{eq:X}
d X^{(n),\boldsymbol{q}}(t)=&\left[\frac{1}{n} \sum_{i=1}^{n} \mu q_{i}(t)-\delta X^{(n),\boldsymbol{q}}(t) \right]	d t+\sigma X^{(n),\boldsymbol{q}}(t) d W(t)-  d A^{(n)}(t).
\end{align}
Here, the abatement process $A^{(n)}=(A^{(n)}(t))_{t\in[0,T]}$ is an $\mathbb{F}$-adapted and non-decreasing continuous process with $A^{(n)}(0)=0$ satisfying a.s., $\forall t\in[0,T]$,
\begin{align}\label{eq:An}
\int_{0}^{t} \mathbf{1}_{\{X^{(n),\boldsymbol{q}}(s)=Z(s)\}} d A^{(n)}(s) =A^{(n)}(t).
\end{align}
In other words, the process $A^{(n)}=(A^{(n)}(t))_{t\in[0,T]}$ increases strictly on $\{t\in[0,T];~X^{(n),\boldsymbol{q}}(t)=Z(t)\}$ only. Then, we have $X^{(n),\boldsymbol{q}}(t)\leq Z(t)$, a.s. for all $t\in[0,T]$.

\subsection{The Objectives of Regional Administrators and Central Regulator}

To introduce the objectives of regional administrators and the central regulator, we begin by specifying the admissible control set of the leader, denoted by $\mathbb U_0$. This set consists of real-valued adapted processes $u=(u(t))_{t\in[0,T]}$ satisfying $\|u\|_2\leq C_1$, where the constant $C_1$ depends only on $(T,z_0,L,m,M)$ and satisfies $C_1\geq \alpha\, Z(T)^2$ with $\alpha:=\alpha(T,z_0,L,m,M)$ as defined in \eqref{uniformL2_u}. For the followers, we define $\mathbb U_i$ as the set of real-valued adapted processes $q_i=(q_i(t))_{t\in[0,T]}$ such that $\|q_i\|_2\leq C_2$, where $C_2\geq \frac{M^2}{2m^2}T+\frac{\|Z\|_\infty^2}{2m^2}C_1$, and for which the corresponding state process $X^{(n),\boldsymbol q}$ in \eqref{eq:X} remains nonnegative.

\quad  The administrator (follower) of region $i$  seeks to enhance her economic profits, as does each of the other regional administrators, all of whom ultimately maximize their net gains from production.
The instantaneous profit objective of region $i$ under production output processes $\boldsymbol{q}\in\mathbb{U}^{(n)}:=\prod_{i=1}^n\mathbb{U}_i$ at time $t\in[0,T]$ and the assumption that the outputs of regions are homogeneous goods, can be expressed as:
\begin{align}\label{eq:Pit}
\Pi^{(n),\boldsymbol{q},u}_{i}(t):=P_{i}^{(n),\boldsymbol{q},u}(t) q_{i}(t)-{c}_{i}\left|q_{i}(t)\right|^{2},
\end{align}
where ${c}_i>0$ is the production cost parameter, and hence the cost of products is depicted by ${c}_{i}|q_{i}(t)|^2$ (\citealt{yeung2007dynamically}; \citealt{yeung2008cooperative}); while $P_i^{(n),\boldsymbol{q},u}(t)$ is the price of the output in region $i$ at time $t$, which obeys the dynamics (\citealt{hoel2001taxes,golosov2014optimal}): 
\begin{align}\label{eq:Pi}
P_{i}^{(n),\boldsymbol{q},u}(t)=a_{i}+u(t) X^{(n),\boldsymbol{q}}(t),\quad \forall t\in[0,T],
\end{align}
where $a_i>0$ is the stable price of goods in region $i$ and $u=(u(t))_{t\in[0,T]}\in\mathbb{U}_0$ as an $\mathbb{F}$-adapted process is the consistent price adjustment measure, which is determined by the central regulator. In lieu of \eqref{eq:Pit}, by choosing the production strategy $q_{i}=(q_{i}(t))_{t\in[0,T]}$, the administrator of region $i$ aims to maximize the region’s net profit given by, for $i=1,\ldots,n$,
\begin{align}\label{eq:Ji}
&J_{i}^{(n)}(\boldsymbol{q};u)=\mathbb{E}\left[\int_{0}^{T} \Pi^{(n),\boldsymbol{q},u}_{i}(t) dt\right]\\
&=\mathbb{E}\left[\int_{0}^{T} \left(\left(a_i +u(t) X^{(n),\boldsymbol{q}}(t) \right) q_{i}(t)-c_{i}\left| q_{i}(t)\right|^{2}\right) dt \right].\nonumber
\end{align}
Here, for $i\geq1$, $\boldsymbol{\theta}_{i}:=(a_{i}, {c}_{i})$ is called the type vector. We here assume that the type vectors $(\boldsymbol{\theta}_i)_{i\geq1}$ fall into a compact set $\mathcal{O}:=[m,M]^2$ with $0<m<M<\infty$.

\quad The central regulator (leader) seeks to achieve the emission reduction target by adjusting product prices. Accordingly, the central regulator aims to minimize the following objective functional given by
\begin{align}\label{eq:J0}
J_0^{(n)}(u;\boldsymbol{q})=&\mathbb{E}\bigg[\int_{0}^{T} c_{\rm P} (u(t))^2 d t+c_{\rm A}\int_{0}^{T} d A^{(n)}(t)+c_{\rm E}X^{(n),\boldsymbol{q}}(T) \bigg],
\end{align}
where $c_{\rm P}>0$ is the relative importance assigned to price stability, $c_{\rm A}>0$ refers to as the abatement cost coefficient that characterizes the relationship between abatement inputs and the resulting economic costs and $c_{\rm E}>0$ denotes the intensity of the penalty borne by the central regulator resulting from emission pollution at terminal horizon. 
 For notational simplicity, we assume that each of the remaining parameters $(\mu,\sigma,\delta,c_{\rm P},c_{\rm A},c_{\rm E})$ lies in the compact set $[m,M]$.

\subsection{Stackelberg Game Problem}

In this subsection, we define the approximate Stackelberg equilibrium associated with the objective functionals \eqref{eq:Ji} and \eqref{eq:J0} for the followers and leader. The game among followers is formulated as a finite $n$-player mean field game (MFG). We then first present the definition of the approximate Nash equilibrium for this $n$-follower game, under a given strategy adopted by the leader.

\begin{definition}[Approximate Nash Equilibrium (ANE)]\label{def:Nash-equil} 
Given a price adjustment policy $u=(u(t))_{t\in[0,T]}$ $\in\mathbb{U}_0$ implemented by the central regulator (leader), the vector of strategies $\boldsymbol{q}^{*,(n)}(u)=(q^{*,(n)}_{1}(u),\ldots,$ $q^{*,(n)}_{n}(u))\in\mathbb{U}^{(n)}$ is called an $\epsilon$-Nash equilibrium of the finite $n$-region game \eqref{eq:Ji} and \eqref{eq:X}, if there is $\epsilon=\epsilon(n)>0$ satisfying $\lim_{n\to\infty} \epsilon(n)=0$ such that, for all $i=1,\ldots,n$,
\begin{align*}
J_{i}^{(n)}(\boldsymbol{q}^{*,(n)}(u))\geq &\sup_{q_{i}\in\mathbb{U}_i} J_{i}^{(n)}(q_{i},\boldsymbol{q}^{*,(n),-i}(u))-\epsilon
\end{align*}	
with the vector of strategies $\boldsymbol{q}^{*,(n),-i}(u):=(q^{*,(n)}_{1}(u),\ldots,$ $q^{*,(n)}_{i-1}(u),$ $q^{*,(n)}_{i+1}(u),\ldots,q^{*,(n)}_{n}(u) )$.
\end{definition}

\quad Based on the approximate Nash equilibrium for a fixed leader's strategy in Definition \ref{def:Nash-equil}, we now define the approximate Stackelberg equilibrium for the leader and followers.
\begin{definition}[Approximate Stackelberg Equilibrium (ASE)]\label{def:SE}
Let $u^{*,(n)}\in\mathbb{U}_0$ be the leader's strategy and $\boldsymbol{q}^{*,(n)}(u^{*,(n)})=(q_1^{*,(n)}(u^{*,(n)}),\dots,q_n^{*,(n)}(u^{*,(n)}) )\in\mathbb{U}^{(n)}$ be the vector of strategies implemented by $n$ followers. Then, the vector of strategies $(u^{*,(n)},q_1^{*,(n)}(u^{*,(n)}),\dots,$ $q_n^{*,(n)}(u^{*,(n)}))$ is an $(\epsilon_1,\epsilon_2)$-Stackelberg equilibrium with respect to $(J_i^{(n)})_{i=0}^n$ if there exist $(\epsilon_1,\epsilon_2)=(\epsilon_1(n),\epsilon_2(n))\in\R_+^2$ satisfying $\lim_{n\to\infty}\epsilon_i(n)=0$ ($i=1,2$) such that
\begin{itemize}
\item[{\rm(i)}] $\boldsymbol{q}^{*,(n)}(u)=(q_1^{*,(n)}(u),\ldots,q_n^{*,(n)}(u))$ constitutes an $\epsilon_1$-Nash equilibrium under any $u\in \mathbb{U}_{0}$;
\item[{\rm(ii)}] It holds that  $ J_0^{(n)}(u^{*,(n)};\boldsymbol{q}^{*,(n)}(u^{*,(n)}))\leq$ $ \inf_{u\in\mathbb{U}_0} J_0^{(n)}(u;$ $\boldsymbol{q}^{*,(n)}(u))+\epsilon_2$.
\end{itemize}
\end{definition}

\section{Approximate Nash Equilibrium for Followers}\label{sec:non-MFG}

In Stackelberg games, the leader’s strategy is determined first, as it is based on the anticipation of the follower’s
optimal response. This sequential decision-making process arises from the fact that the leader’s choice influences the follower’s behavior, and the follower’s reaction must be considered when optimizing the leader’s strategy. Therefore, we apply the method of backward induction in solving this game problem for one leader and $n$ followers.

\quad In this section, we solve the followers' problem \eqref{eq:X} and \eqref{eq:Ji} by using the MFG approach for $n$ followers under a given consistent price adjustment strategy $u\in\mathbb{U}_0$ adopted by the leader. To do it, recall the type vector $\boldsymbol{\theta}_{i}=(a_i,c_i)\in \mathcal{O}$ of region $i=1,\ldots,n$, it induces a sequence of empirical measures $\nu^{(n)}:=\frac{1}{n}\sum_{i=1}^n\delta_{\boldsymbol{\theta}_i}$ for $n\geq1$. The following assumption is required to formulate the MFG problem:
\begin{itemize}
\item[$(\mathbf{A_{1}})$] there is a $\mathcal{O}$-valued r.v. $\boldsymbol{\theta}:=(a,c)$ on the probability space $(\Omega,\mathbb{F},\mathbb{P})$  independent of common Brownian motion $W$  with the law $\nu \in \mathcal{P}(\mathcal{O})$ such that $\nu^{(n)}(f)\to\nu(f)$ as $n\to\infty$ for all $f\in C_b(\mathcal{O})$.
\end{itemize}

\subsection{MFG for Representative Follower}

Let $\mathbb{F}^{\rm MF}=(\mathcal{F}^{\rm MF}_{t})_{t\in[0,T]}$ with $\mathcal{F}^{\rm MF}_{t}:=\sigma(W(s);~s\leq t)$ and $\mathbb{U}^{\rm MF}$ be the admissible control set consisting of $\mathbb{F}^{\rm MF}$-adapted real-valued processes $q=(q(t))_{t\in[0,T]}$ satisfying $\|q\|_2\leq C_2$. 
Then, the MFG problem for a representative follower is formulated as, for any continuous process $m_{q}^{u}=(m_{q}^{u}(t))_{t\in[0,T]}\in L_{\mathbb{F}^{\rm MF}}^2(0,T;\R)$,
\begin{align}\label{eq:MFG-followers}
\sup_{q\in \mathbb{U}^{\rm MF}}\bar{J}(q,u;m_{q}^{u})
& =\sup_{q\in \mathbb{U}^{\rm MF}}\mathbb{E}\left[\int_{0}^{T}\Big(\left(a+u(t) X^q(t;m_{q}^{u}) \right) q(t)-{c}\left|q(t)\right|^2\Big)dt\right]\\[0.4em]
\text{s.t.}~ X^q(t;m_{q}^{u})=x_0&+\int_0^t\left(\mu m_{q}^{u}(s)-\delta X^q(s;m_{q}^{u})\right) ds\nonumber+\int_0^t\sigma X^q(s;m_{q}^{u}) d W(s)-A(t)\leq Z(t),~~\text{a.s.,}\nonumber\\[0.4em]
&~\text{$t\mapsto A(t)$ is continuous and non-decreasing with } A(0)=0\nonumber\\[0.4em]
&\qquad~\text{and}~\int_0^t {\bf1}_{\{X^{q}(s;m_{q}^{u})=Z(s)\}}dA(s)=A(t),~\forall t\in[0,T].\nonumber
\end{align}
Intuitively, conditional on $\mathbb{F}^{\rm MF}$, the given continuous process $m_{q}^{u}=(m_{q}^{u}(t))_{t\in[0,T]}$ can be viewed as the limit of mean field term $\frac{1}{n}\sum_{i=1}^n q_i(t)$ for $t\in[0,T]$ when the number of regions $n$ is sufficiently large. This approximation can be heuristically justified by the law of large numbers. Hence, the limiting process $m_{q}^{u}$ is intrinsically determined by control processes. A key challenge is how to explicitly characterize limiting process $m_{q}^{u}$ so as to satisfy the (conditional) consistency condition. This is crucial for solving the above optimal control problem for the representative region, as the resulting best-response solution is necessary for deriving the corresponding fixed-point problem. This results in the so-called mean field equilibrium (MFE).

\quad For any continuous process $m_{q}^{u}=(m_{q}^{u}(t))_{t\in[0,T]}\in L_{\mathbb{F}^{\rm MF}}^2(0,T;\R)$, the SDE in~\eqref{eq:MFG-followers} always admits a unique solution $(X^{q}, A)=(X^{q}(t;m_q^u), A(t))_{t\in[0,T]}$, and hence the cost functional $\bar{J}(q,u;m_{q}^{u})$ given in \eqref{eq:MFG-followers} is well-defined. We next introduce the so-called dynamic Skorokhod problem  (c.f., Definition 2.4 in \citealt{BWY2025}):
\begin{definition}[Dynamic Skorokhod Problem]\label{def:Skorokhod}
Let $\tilde{x}=(\tilde{x}(t))_{t\in[0,T]}\in \mathcal{C}_T$ and $z=(z(t))_{t\in[0,T]}\in\mathcal{C}_T^+$ satisfy $\tilde{x}(0)\leq z(0)$. A pair of functions $(x,g)\in \mathcal{C}_T\times \mathcal{C}_T^{+}$ is called a solution to the dynamic Skorokhod problem for $(\tilde{x},z)$ if it holds that
\begin{enumerate}[{\rm(i)}]
\item $x(t):=\tilde{x}(t)-g(t)\leq z(t)$ for all $t\in[0,T]$;
\item $g(0)=0$ and $t\mapsto g(t)$ is continuous and non-decreasing;
\item $\int_{0}^{t} \mathbf{1}_{\{x(s)=z(s)\}} d g(s)=g(t)$ for $t\in[0,T]$.
\end{enumerate}
\end{definition}
Then, it follows from \cite{piliPenko2014} that the solution of dynamic Skorokhod problem given in Definition \ref{def:Skorokhod}  has the following form given by
\begin{lemma}\label{lem:Skorokhod}
Let $\tilde{x}=(\tilde{x}(t))_{t\in[0,T]}\in \mathcal{C}_T$ and $z=(z(t))_{t\in[0,T]}\in \mathcal{C}_T^+$ satisfy $\tilde{x}(0)\leq z(0)$. Then, there exists a unique solution $(x,g)\in \mathcal{C}_T\times\mathcal{C}_T^+$ to the dynamic Skorokhod problem $(\tilde{x},z)$. Furthermore, it holds that, for any $t\in[0,T]$, 
\begin{align*}
x(t)=\tilde{x}(t)-\sup_{s\in[0,t]}\left(\tilde{x}(s)-z(s)\right)^{+},
g(t)=\sup_{s\in[0,t]}\left(\tilde{x}(s)-z(s)\right)^{+}
\end{align*}
with $x^+=\max\{x,0\}$ for $x\in\R$.
\end{lemma}

\quad Next, we introduce the set of functions $\mathcal{D}:=\{(\tilde{x},z)\in\mathcal{C}_T\times\mathcal{C}_T^+;~ \tilde{x}(0)\leq  z(0)\}$ and define the dynamic Skorokhod mapping $\Gamma:\mathcal{D}\to\mathcal{C}_T$ by $\Gamma(\tilde{x},z)=x$, where $(x,g)$ is a solution to the dynamic Skorokhod problem $(\tilde{x},z)$. It follows from Lemma~\ref{lem:Skorokhod} that $\Gamma:\mathcal{D}\to\mathcal{C}_T$ is Lipchitz continuous, i.e., for any $(\tilde{x}_{i},z_{i})\in\mathcal{D}$ with $i=1,2$,
\begin{align}\label{Skorohod-Lip}
&\left\|\Gamma(\tilde{x}_1,z_1)-\Gamma(\tilde{x}_2,z_2)\right\|_{\infty}\leq  2\left\|\tilde{x}_1-\tilde{x}_2\right\|_{\infty}+\left\|z_1-z_2\right\|_{\infty}.
\end{align}
We will use the mapping $\Gamma:\mathcal{D}\to\mathcal{C}_T$ to represent the reflected state process in \eqref{eq:MFG-followers}. To do it, let us introduce the process $\widetilde{X}^q=(\widetilde{X}^q(t))_{t\in[0,T]}$ which satisfies the dynamics:
\begin{align}\label{widetilde-X}
&d\widetilde{X}^q(t;m_{q}^{u})=\left(\mu m_{q}^{u}(t)-\delta X^{q}(t;m_{q}^{u})\right) d t+\sigma X^{q}(t;m_{q}^{u}) d W(t),~~\widetilde{X}^q(0;m_{q}^{u})=x_0\leq z_0.    
\end{align}
Then, $(X^q,A)=(X^q(t;m_{q}^{u}),A(t))_{t\in[0,T]}$ is the unique solution to the dynamic Skorokhod problem $(\widetilde{X}^{q},Z)=(\widetilde{X}^{q}(t;m_{q}^{u}),Z(t))_{t\in[0,T]}$ in the sense of Definition \ref{def:Skorokhod}. Consequently, $X^{q}=\Gamma(\widetilde{X}^q,Z)$ with the upper boundary $Z=(Z(t))_{t\in[0,T]}$ being given by \eqref{eq:E0-SDE} and the local time process being given by $A(t)=\sup_{s\leq t}(\widetilde{X}^{q}(s)-Z(s))^+$ for $t\in[0,T]$.

\subsection{Mean Field Equilibrium}

This section examines the best-response solution of the followers' representative control problem \eqref{eq:MFG-followers} subject to the leader's consistent price adjustment strategy $u\in\mathbb{U}_0$, and then establishes the corresponding mean field equilibrium (MFE).

\quad By solving \eqref{eq:MFG-followers} using the dynamic programming principle (DPP) and Lemma \ref{lem:Skorokhod}, we have the following lemma whose proof is omitted. 
\begin{lemma}\label{lem:bestres-F}
Let the leader's strategy $u\in \mathbb{U}_0$ and {a fixed continuous process} $m_{q}^{u}=(m_{q}^{u}(t))_{t\in[0,T]}\in L^2_{\mathbb{F}^{\rm MF}}(0,T;\R)$. Then, there exists a unique best production response strategy $q^{*,u}=(q^{*,u}(t))_{t\in[0,T]}$ for the representative follower's problem \eqref{eq:MFG-followers}, which is given by $q^{*,u}(t)= (2c)^{-1}(a+u(t) X^{*,u}(t;m_{q}^{u}))$ for $t\in[0,T]$, where the underlying state process $X^{*,u}=(X^{*,u}(t;m_{q}^{u}))_{t\in[0,T]}$ is given by, for all $t\in[0,T]$,
\begin{align}\label{eq:star-X-u-0}
&X^{*,u}(t;m_{q}^{u})=x_0+\int_0^t\left(\mu m_{q}^{u}(s)-\delta X^{*,u}(s;m_{q}^{u})\right) d s+\int_0^t\sigma X^{*,u}(s;m_{q}^{u}) d W(s) -A^{*,u}(t;m_{q}^{u})\leq Z(t).
\end{align}
Here, $A^{*,u}=(A^{*,u}(t;m_{q}^{u}))_{t\in[0,T]}$ is continuous and non-decreasing process which also satisfies that
\begin{align*}
A^{*,u}(t;m_{q}^{u})=&\sup_{s\leq t}\left[x_0-z_0+\int_{0}^{s} \big(\mu m_{q}^{u}(r)-\delta X^{*,u}(r;m_{q}^{u})-\Phi(Z(r))\big)d r+\sigma\int_{0}^{s} X^{*,u}(r;m_{q}^{u})  d W(r)\right]^{+}.    
\end{align*}
\end{lemma}

\vspace{-0.4cm}
\quad Based on the best production response strategy $q^{*,u}=(q^{*,u}(t))_{t\in[0,T]}$ given in Lemma~\ref{lem:bestres-F}, we next provide the definition of MFE given by 
\begin{definition}[MFE]\label{def:MFE}
For a given leader's strategy $u\in\mathbb{U}_0$, let $q^{*,u}=(q^{*,u}(t))_{t\in[0,T]}\in \mathbb{U}^{\rm MF}$ be an admissible strategy. We say that $q^{*,u}$ is an MFE if $q^{*,u}$ is a best response solution for problem \eqref{eq:MFG-followers} corresponding to the choice of a continuous process $m_{q}^{u}=(m_{q}^{u}(t))_{t\in[0,T]}\in L^2_{\mathbb{F}^{\rm MF}}(0,T;\R)$ in~\eqref{eq:star-X-u-0}, denoted by ${m}^{*,u}_q$, satisfying the (conditional) consistency condition $\mathbb{E}[q^{*,u}(t)|\mathcal{F}_t^{\rm MF}]={m}^{*,u}_{q}(t)$ for all $t\in[0,T]$.
\end{definition}

\quad Then, we have
\begin{lemma}
For a given leader's strategy $u\in\mathbb{U}_0$, the MFE in the sense of Definition~\ref{def:MFE} is given by $q^{*,u}(t)=\frac{a -b m_{q}^{*,u}(t)}{2 {c}}$ for $t\in[0,T]$, where the equilibrium process $m_{q}^{*,u}=(m_{q}^{*,u}(t))_{t\in[0,T]}$ has the form, for $t\in[0,T]$,
\begin{align}\label{cc}
m_{q}^{*,u}(t)&=\nu\left(\frac{a}{2c}\right)+\nu\left(\frac{1}{2c}\right) u(t) X^{*,u}(t),
\end{align}
where we recall that $\nu\left(\frac{a}{2c}\right):=\int_{\mathcal{O}}\frac{a}{2{c}}\nu(d\boldsymbol{\theta})$ and $\nu\left(\frac{1}{2c}\right):=\int_{\mathcal{O}}\frac{1}{2{c}}\nu(d\boldsymbol{\theta})$ with $\boldsymbol{\theta}=(a,c)$. 
Here, $X^{*,u}=(X^{*,u}(t))_{t\in[0,T]}$ is the emission process at the equilibrium given by, for $t\in[0,T]$,
\begin{align}\label{eq:star-X-u-1}
X^{*,u}(t)&=x_0+\int_0^t\left[\frac{\mu}{2} \nu\left(\frac{a}{c}\right)+\left(\frac{\mu}{2} \nu\left(\frac{1}{c}\right) u(s) -\delta \right)X^{*,u}(s)\right] d s\nonumber\\
&\quad+\sigma\int_0^t X^{*,u}(s) d W(s)-A^{*,u}(t)\leq Z(t),  
\end{align}
where the local time process $A^{*,u}(t)=\sup_{s\leq t}[z_0-x_0+\int_{0}^{s} (\Phi(Z(r))-\frac{\mu}{2}\nu ( \frac{a}{ {c}} ) - (\frac{\mu}{2} \nu (\frac{1}{ {c}} )u(r) -\delta ) X^{*,u}(r) ) dr$ $-\int_{0}^{s}\sigma X^{*,u}(r) d W(r)]^{-}$ for $t\in[0,T]$.
\end{lemma}
\vspace{-0.4cm}
\begin{proof}
Using Lemma~\ref{lem:bestres-F} and Definition \ref{def:MFE}, it suffices to find a continuous process $m_{q}^{*,u}=(m_{q}^{*,u}(t))_{t\in[0,T]}\in L_{\mathbb{F}^{\rm MF}}^2(0,T;\R)$ which satisfies the consistent condition:
\begin{align}\label{eq:mqstart}
m_{q}^{*,u}(t)&=\mathbb{E}\left[q^{*,u}(t)\big|\mathcal{F}_t^{\rm MF}\right]=\mathbb{E}\left[(2{c})^{-1}\left(a+u(t) X^{*,u}(t;m_{q}^{u})\right)\big|\mathcal{F}_t^{\rm MF}\right]\\
&=\int_{\mathcal{O}}\frac{a}{2{c}}\nu(d\boldsymbol{\theta})+ u(t) X^{*,u}(t;m_{q}^{u})\left(\int_{\mathcal{O}} \frac{1}{2{c}}\nu(d\boldsymbol{\theta})\right)\nonumber\\
 &=\nu\left(\frac{a}{2c}\right)+\nu\left(\frac{1}{2c}\right) u(t) X^{*,u}(t;m_{q}^{u}),\quad \forall t\in[0,T].\nonumber
\end{align}
The last equality holds since both $u(t)$ and $X^{*,u}(t;m_{q}^{u})$ are $\mathcal{F}_t^{\rm MF}$-measurable for any $t\in[0,T]$. Plugging $m_{q}^{*,u}=(m_{q}^{*,u}(t))_{t\in[0,T]}$ given by \eqref{eq:mqstart} into \eqref{eq:star-X-u-0}, we obtain $X^{*,u}=(X^{*,u}(t;m_{q}^{*,u}))_{t\in[0,T]}$ satisfies the dynamics \eqref{eq:star-X-u-1}. This yields that the equilibrium process $m_{q}^{*,u}(t)$ for $t\in[0,T]$ determined by \eqref{eq:mqstart} has the form given by \eqref{cc}.  
\end{proof}

\subsection{ANE for Finite $n$ Followers}

In this section, we establish an approximate Nash equilibrium of the finite $n$-follower game  problem \eqref{eq:Ji} and \eqref{eq:X} under a given consistent price adjustment strategy $u\in\mathbb{U}_0$ adopted by the leader.

\quad The main result on the approximate Nash equilibrium of finite $n$-follower game problem \eqref{eq:Ji} and \eqref{eq:X} is stated as follows:
\begin{theorem}\label{thm:eps-Nash}
Let Assumption $({\bf A_1})$ hold and $u\in\mathbb{U}_0$ be a strategy adopted by the leader. Let us introduce that, for $i=1,\ldots,n$, 
\begin{align}\label{eq:n-opt-q}
q_i^{*,(n)}(t;u)= \frac{a_i+u(t) X^{*,u,(n)}(t)}{2c_i},\quad \forall t\in[0,T].
\end{align}
Here, the underlying state process $X^{*,u,(n)}=(X^{*,u,(n)}(t))_{t\in[0,T]}$ satisfies that $X^{*,u,(n)}(0)=x_0$ and for $t\in(0,T]$,
\begin{align}\label{eq:n-opt-X-F}
d X^{*,u,(n)}(t)=&\left[\left(\frac{\mu}{2} \nu^{(n)}\left(\frac{1}{c}\right)u(t) -\delta\right) X^{*,u,(n)}(t)+\frac{\mu}{2} \nu^{(n)}\left(\frac{a}{c}\right) \right] d t\nonumber\\
&+\sigma X^{*,u,(n)}(t) d W(t)-d  A^{*,u,(n)}(t),
\end{align}
where we recall that $\nu^{(n)}=\frac{1}{n}\sum_{i=1}^n\delta_{\boldsymbol{\theta}_i}$ is the empirical measure associated with the type vectors and for $t\in[0,T]$,
\begin{align*}
A^{*,u,(n)}(t)&=\sup_{s\leq t}\Bigg[z_0-x_0+\int_{0}^{s} \left(\Phi(Z(r))-\frac{\mu}{2}\nu^{(n)}\left( \frac{a}{ {c}}\right) -\left(\frac{\mu}{2} \nu^{(n)}\left(\frac{1}{ {c}}\right)u(r) -\delta\right) X^{*,u,(n)}(r)\right) dr\\
&\quad-\int_{0}^{s}\sigma X^{*,u,(n)}(r) d W(r)\Bigg]^{-}.
\end{align*}
Then, the vector of strategies $\boldsymbol{q}^{*,(n)}(u)=(q_1^{*,(n)}(u),\ldots,$ $q_n^{*,(n)}(u) )\in\mathbb{U}^{(n)}$ constitutes an $\epsilon$-Nash equilibrium for some $\epsilon=\epsilon(n)$ satisfying $\lim_{n\to\infty}\epsilon(n)=0$.
\end{theorem}

\quad The proof of Theorem \ref{thm:eps-Nash} relies on the following auxiliary result.
\begin{lemma}\label{lem:X-star-positive}
The state process $X^{*,u,(n)}=(X^{*,u,(n)}(t))_{t\in[0,T]}$ given in \eqref{eq:n-opt-X-F} satisfies that $X^{*,u,(n)}(t) > 0$, a.s. for all $t\in [0,T]$.
\end{lemma}
\vspace{-0.4em}
\begin{proof}
Define $\tau:= \inf \{ t \in [0,T];~X^{*,u,(n)}(t) \leq 0 \}\wedge T$ with the convention $\inf\varnothing = +\infty$.
We argue by contradiction and assume $\Pb(\tau<T)>0$. Let $G=(G(t))_{t\in [0,T]}$ be the GBM satisfying $d G(t)=(\delta+\sigma^2-\frac{\mu}{2} \nu^{(n)} (\frac{1}{ {c}} )u(t)) G(t) d t-\sigma G(t) d W(t)$ with $G(0)=1$. The It\^o's formula yields that 
\begin{align}\label{tilde_G}
&d\left(X^{*,u,(n)}(t)G(t)\right)=\frac{\mu}{2}\nu^{(n)}\left(\frac{a}{ {c}}\right) G(t) d t
-G(t) d A^{*,u,(n)}(t).
\end{align}
By the continuity of $t\mapsto X^{*,u,(n)}(t)$, we have
$X^{*,u,(n)}(\tau )=0<Z^{(n)}(\tau)$  on $\{\tau<T\}$.
Consequently, for every $\omega\in\{\tau<T\}$,
there exists $\delta(\omega)>0$ such that $0<X^{*,u,(n)}(s,\omega)<Z^{(n)}(t,\omega)$ when $s\in[\tau(\omega)-\delta(\omega),\tau(\omega))$.
Introduce $\tilde G(t):=X^{*,u,(n)}(t)G(t)$ for $t\in[0,T]$. Then, $\tilde G(t)>0$ on $t\in[0,\tau)$. 
Integrating \eqref{tilde_G} over
$[\tau-\delta,\tau]$ gives that
\begin{align}\label{tau_delta-G}
&\tilde G(\tau)-\tilde G({\tau-\delta})=\int_{\tau-\delta}^{\tau}\frac{\mu}{2}\nu^{(n)}\left(\frac{a}{ {c}}\right)G(t)dt-\int_{\tau-\delta}^{\tau}G(t)d A^{*,u,(n)}(t),~~ \text{on } \{\tau<T\}.
\end{align}
\quad By the construction of $t\mapsto A^{*,u,(n)}(t)$, the reflecting behavior is inactive on
$[\tau-\delta,\tau]$, i.e., $d A^{*,u,(n)}(t)=0$ a.s. on this interval. Hence, by the definition of $\tau$,  one has $\tilde G({\tau})-\tilde G({\tau-\delta})=\int_{\tau-\delta}^{\tau}        \frac{\mu}{2}\nu^{(n)}\left(\frac{a}{ {c}}\right)G(t) dt>0$ on $\{\tau<T\}$. This contradicts the facts that
$\tilde G(\tau)=0$ and $\tilde G(\tau-\delta)>0$
on $\{\tau<T\}$. Thus, $\tau=T$ a.s..

\quad We next prove that $X^{*,u,(n)}(T)>0$ a.s.. Suppose otherwise that $\Pb(X^{*,u,(n)}(T)=0)>0$. Repeating the above argument on the event $\{X^{*,u,(n)}(T)=0\}$, one can construct $\delta>0$ such that $\tilde G(T)-\tilde G(T-\delta)=\int_{T-\delta}^{T}\frac{\mu}{2}\nu^{(n)}\left(\frac{a}{ {c}}\right) G(t)dt>0$ which contradicts that fact that $\tilde G(T)=0$. Hence, $X^{*,u,(n)}(T)>0$ a.s.. This immediately yields that $\tau=T$. Furthermore, since 
$X^{*,u,(n)}(T)>0$ a.s., we conclude the desired result.
\end{proof}
We now proceed to the proof of Theorem \ref{thm:eps-Nash}.
\vspace{-0.4em}
\begin{proof}[Proof of Theorem \ref{thm:eps-Nash}]
We introduce the following auxiliary stochastic control problem, for a fixed leader's strategy $u=(u(t))_{t\in[0,T]}\in\mathbb{U}_0$, 
\begin{align}\label{eq:aux-MFG-followers}
\sup_{q\in \mathbb{U}^{\rm MF}}\bar{J}_i(q;u)&:= \mathbb{E}\bigg[\int_{0}^{T} e^{-\rho t}\big(\left(a_i+u(t) \bar{X}^u(t)\right) q(t)- c_i(q(t))^2\big)dt\bigg],
\end{align}
subject to a.s.
\begin{align*}
\bar{X}^{u}(t)=&x_0+\int_{0}^{t}\bigg[\left(\mu \nu\left(\frac{1}{2 {c}}\right) u(s) -\delta \right)\bar{X}^{u}(s)+\mu \nu\left(\frac{a}{2 {c}}\right)\bigg]ds\nonumber\\
&+\int_{0}^{t}\sigma \bar{X}^{u}(s)d W(s)- \bar{A}(t)\leq Z(t),~~\forall t\in[0,T], 
\end{align*}
and $t\mapsto \bar{A}(t)$ is continuous and non-decreasing with $\bar{A}(0)=0$ and $\int_0^t {\bf 1}_{\{\bar{X}(s)=Z(s)\}}d\bar{A}(s)=\bar{A}(t)$ for $t\in[0,T]$. 

\quad By the DPP and solving the corresponding HJB equation of the auxiliary problem \eqref{eq:aux-MFG-followers}, the optimal (feedback) control strategy of problem \eqref{eq:aux-MFG-followers} is given by, for $i=1,\ldots,n$ and for any $t\in[0,T]$, 
\begin{align}\label{eq:opt-q-bar-i}
  \bar{q}_i^{*,u}(t)=\bar{q}_i^{*,u}(\bar{X}^u(t)):=\frac{a_i+ u(t) \bar{X}^{u}(t) }{2 {c}_i}.
\end{align}
Following the proof of Lemma~\ref{lem:X-star-positive}, we have $\bar{X}^u(t)>0$, a.s. for all $t\in[0,T]$. Moreover, it is not difficult to show that {(the proof may refer to Appendix \ref{sec:proof})}:
\begin{align}\label{eq:momentestite0}
\lim_{n\to\infty}  \mathbb{E}\left[\left\|X^{*,u,(n)}-\bar{X}^u\right\|_{\infty}^2\right]=0.
\end{align}
For $i=1,\ldots,n$, consider the vector of strategies $\boldsymbol{q}^{*,(n),-i}(u):=(q^{*,(n)}_{1}(u),\ldots,q^{*,(n)}_{i-1}(u),q^{*,(n)}_{i+1}(u),\ldots,q^{*,(n)}_{n}(u) )$ with $q^{*,(n)}_{j}(u)=(q^{*,(n)}_{j}(t;u))_{t\in[0,T]}$ being given by \eqref{eq:n-opt-q} for $j\neq i$. Let $X^{*,u,(n)}_{-i}=(X^{*,u,(n)}_{-i}(t))_{t\in[0,T]}$ be the state process under the vector of strategies $(q_i,\boldsymbol{q}^{*,(n),-i}(u))\in \mathbb{U}^{(n)}$ satisfying the following dynamics:
\begin{align}\label{eq:n-opt-X-F-i}
&d X^{*,u,(n)}_{-i}(t)=\left[\frac{\mu}{n} q_i(t)+\mu \frac{1}{n}\sum_{j\neq i}^{n}q_i^{*,(n)}(t;u)-\delta X^{*,u,(n)}_{-i}(t)\right]d t+\sigma X^{*,u,(n)}_{-i}(t) d W(t)-d  A^{*,u,(n)}_{-i}(t)\nonumber\\
&\quad\qquad\qquad=\left[\frac{\mu}{n} q_i(t) +\mu \frac{n-1}{n}\nu^{(n)}_{-i}\left(\frac{1}{2 {c}}\right)u(t) X^{*,u,(n)}(t) +\mu \frac{n-1}{n}\nu^{(n)}_{-i}\left(\frac{a}{2 {c}}\right)-\delta X^{*,u,(n)}(t)\right] d t\nonumber\\
&\quad\qquad\qquad\quad+\sigma X^{*,u,(n)}_{-i}(t) d W(t) -d  A^{*,u,(n)}_{-i}(t),
\end{align}
where the reflecting term $A^{*,u,(n)}_{-i}(t)$ for $t\in[0,T]$ is given by $A^{*,u,(n)}_{-i}(t)=\sup_{s\leq t}[z_0-x_0+\int_{0}^{s} (\Phi(Z(r))-\frac{\mu}{n} q_i(r) -(\mu \frac{n-1}{n}\nu^{(n)}_{-i}\left(\frac{1}{2 {c}}\right)u(r) -\delta) X^{*,u,(n)}(r)) dr -\mu \frac{n-1}{n}\nu^{(n)}_{-i}(\frac{a}{2 {c}}) s-\int_{0}^{s}\sigma X^{*,u,(n)}_{-i}(r) d W(r)]^{-}$. Here, $\nu^{(n)}_{-i}(f):=\frac{1}{n-1}\sum_{j\neq i}^{n} f(\boldsymbol{\theta}_j)$ with $\boldsymbol{\theta}_j=(a_j,c_j)$ and  $X^{*,u,(n)}=(X^{*,u,(n)}(t))_{t\in[0,T]}$ is given by~\eqref{eq:n-opt-X-F}. Then, we have {(the proof may refer to Appendix \ref{sec:proof})}:
\begin{align}\label{eq:sup-i-bar}
\lim_{n\to\infty}  \mathbb{E}\left[\left\|X_{-i}^{*,u,(n)}-\bar{X}^u\right\|_{\infty}^2\right]=0.
\end{align}
\quad On the other hand, by Lemma \ref{lem:X-star-positive}, it follows that $0<X^{*,u,(n)}(t)\leq Z(t)$ a.s. for all $t\in[0,T]$. This yields that $(X^{*,u,(n)}(t))^2\leq |Z(t)|^2$ a.s. for all $t\in[0,T]$. In view of \eqref{eq:n-opt-q}, we have, for $i=1,\dots,n$,
\begin{align*}
&\mathbb{E}\left[ \int_{0}^{T} \left|q_i^{*,(n)}(t;u)\right|^2dt \right] \leq \frac{M^2}{2m^2}T + \frac{1}{2m^2} \mathbb{E}\left[\int_{0}^{T} \left|u(t)X^{*,u,(n)}(t)\right|^2dt \right]\\
&\leq \frac{M^2}{2 m^2}T + \frac{ \|Z\|_{\infty}^2 }{2 m^2}\|u\|_2^2\leq \frac{M^2}{2 m^2}T + \frac{ \|Z\|_{\infty}^2 }{2 m^2}C_1\leq C_2. 
\end{align*}   
Hence,  $\boldsymbol{q}^{*,(n)}(u)=(q_{1}^{*,(n)}(u),\ldots,q_{n}^{*,(n)}(u) )\in\mathbb{U}^{(n)}$. Furthermore, using the convergences \eqref{eq:momentestite0} and \eqref{eq:sup-i-bar}, together with the Lipschitz continuity of optimal (feedback) control function $x\mapsto \bar{q}^{*,u}(x)$ given by \eqref{eq:opt-q-bar-i}, we can conclude that  $\boldsymbol{q}^{*,(n)}(u)=(q_{1}^{*,(n)}(u),\ldots,q_{n}^{*,(n)}(u) )\in\mathbb{U}^{(n)}$ is an ANE by adopting a similar argument in the proof of Theorem~8.3 of \cite{Lacker}.
\end{proof}

\section{The Leader's Optimal Control Problem with Penalized Reflection}\label{sec:leader}

Due to the nature of the Stackelberg game under consideration, the central regulator (leader) aims to minimize the following objective functional, for $u\in\mathbb{U}_0$,
\begin{align}
J_0^{(n)}(u;\boldsymbol{q}^{*,(n)}(u) )=&\mathbb{E}\bigg[ \int_{0}^{T} c_{\rm P} \left(u(t)\right)^2 d t+c_{\rm A}\int_0^T dA^{*,u,(n)}(t)  +c_{\rm E} X^{*,u,(n)}(T) \bigg],
\end{align}
where the state process $(X^{*,u,(n)},A^{*,u,(n)})=(X^{*,u,(n)}(t),$ $A^{*,u,(n)}(t))_{t\in[0,T]}$ is described as \eqref{eq:n-opt-X-F} with production strategy  $\boldsymbol{q}^{*,(n)}(u)$ given by \eqref{eq:n-opt-q}.

\subsection{The Leader's Value Function and HJB Equation}\label{subsec:LeaderVFHJB}

We define the value function by, for $(t,x,z)\in[0,T]\times\R_+^2$,
\begin{align}\label{eq:vtxz0}
V(t,x,z):=&\inf_{u\in\mathbb{U}_0 }\mathbb{E}_{t,x,z}\bigg[ \int_{t}^{T} c_{\rm P} \left(u(s)\right)^2 d s+c_{\rm A}\int_t^T dA^{*,u,(n)}(s) +c_{\rm E} X^{*,u,(n)}(T)\bigg],
\end{align}
where the conditional expectation operator $\mathbb{E}_{t,x,z}[\cdot]:=\mathbb{E}[\cdot|X^{*,u,(n)}(t)=x,Z(t)=z]$. By using the DPP, the value function $V(t,x,z)$ formally satisfies the following HJB equation with Neumann boundary condition, for $(t,x,z)\in[0,T]\times\R_+^2$,
\begin{align}\label{eq:HJB-L0-R}
&\min_{u\in\mathbb{R}}\bigg[\pa_{x}V(t,x,z)\left(\mu \nu^{(n)}\left(\frac{a}{2 {c}}\right)+\left(\mu \nu^{(n)}\left(\frac{1}{2 {c}}\right) u -\delta \right)x\right)+c_{\rm P}u^2 \bigg]\nonumber\\
&+\pa_{t} V(t,x,z)+\frac{1}{2}\sigma^2 x^2\pa_{xx} V(t,x,z)+\pa_{z}V(t,x,z)\Phi(z)=0 
\end{align}
with $\pa_{x}V(t,x,x)=c_{\rm A}$ for all $(t,x)\in[0,T]\times\R_+$ and $V(T,x,z)=c_{\rm E} x$ for $(x,z)\in\R_+^2$. 

\quad Assume that the value function $V(t,x,z)$ is smooth. Then, we have from the HJB equation \eqref{eq:HJB-L0-R} that, the optimal (feedback) control function is given by, for $(t,x,z)\in[0,T]\times\R_+^2$,
\begin{align}\label{eq:star-u-R}
u^{*,(n)}(t,x,z)=-\frac{\mu}{2 c_{\rm P}}\nu^{(n)}\left(\frac{1}{2 {c}}\right) x \pa_{x}V(t,x,z).
\end{align}
Plugging the feedback control \eqref{eq:star-u-R} into the HJB equation~\eqref{eq:HJB-L0-R} to have that, for $(t,x,z)\in[0,T)\times\R_+^2$,
\begin{align}
\begin{cases}
 \displaystyle  \pa_{t} V(t,x,z)+\pa_{x}V(t,x,z)\left(\mu \nu^{(n)}\left(\frac{a}{2 {c}}\right) -\delta x \right)+\pa_{z}V(t,x,z) \Phi(z)\\[0.4em]
 \displaystyle~-\frac{\mu^2}{4 c_{\rm P}}\left(\nu^{(n)}\left(\frac{1}{2 {c}}\right)\right)^2 x^2\left(\pa_{x} V(t,x,z)\right)^2+\frac{1}{2}\sigma^2 x^2\pa_{xx}^2V(t,x,z) =0,\\[0.4em]
 \displaystyle \pa_{x}V(t,x,x)=c_{\rm A},~~\forall (t,x)\in[0,T]\times\R_+,\\[0.4em]
 \displaystyle V(T,x,z)=c_{\rm E}x,~~\forall (x,z)\in\R_+^2.
\end{cases}   
\end{align}
By applying the Cole-Hopf transformation and we introduce that, for $(t,x,z)\in[0,T]\times\R_+^2$, 
\begin{align*}
\Psi(t,x,z):=\exp\left(-\frac{\mu^2}{2 c_{\rm P}\sigma^2 } \left(\nu^{(n)}\left(\frac{1}{2 {c}} \right)\right)^2 V(t,x,z) \right).  
\end{align*}
Then, the transform $\Psi(t,x,z)$ satisfies the following linear PDE: 
for $(t,x,z)\in[0,T]\times\R_+^2$,
\begin{align}\label{eq:Psi}
&\pa_{t} \Psi(t,x,z)+\pa_{x}\Psi(t,x,z)\left(\mu \nu^{(n)}\left(\frac{a}{2 {c}}\right) -\delta x\right)+\frac{1}{2}\sigma^2 x^2\pa_{xx}\Psi(t,x,z)+\pa_{z}\Psi(t,x,z) \Phi(z)=0
\end{align}
with Robin boundary conditions
\begin{align*}
\pa_{x}\Psi(t,x,x)&=-\frac{c_{\rm A}\mu^2}{2 c_{\rm P}\sigma^2 }\left(\nu^{(n)}\left(\frac{1}{2 {c}}\right)\right)^2\Psi(t,x,x),~~~
\Psi(T,x,z)=\exp\left(-\frac{c_{\rm E} \mu^2 }{2 c_{\rm P}\sigma^2}\left(\nu^{(n)}\left(\frac{1}{2 {c}}\right)\right)^2 x\right).    
\end{align*}

\subsection{The Solvability of Robin Problem \eqref{eq:Psi}}\label{subsec:RobinProblem}

To solve the above Robin problem \eqref{eq:Psi}, we adopt the probabilistic argument. To this purpose, we introduce the following probabilistic representation given by, for $(t,x,z)\in[0,T]\times\R_+^2$,
\begin{align}\label{eq:Psi-R}
&\Psi(t,x,z):=\mathbb{E}\left[\exp\left(-\eta^{(n)}\left( c_{\rm A}R^{t,x,z}(T)+ c_{\rm E}Y^{t,x,z}(T)\right) \right) \right],
\end{align}
where $\eta^{(n)}:=\frac{\mu^2}{2 c_{\rm P}\sigma^2 }\left(\nu^{(n)}\left(\frac{1}{2 {c}}\right)\right)^2 \in [\frac{m^2}{8M^5},\frac{M^2}{8m^5}]$ is uniformly bounded and the state processes $(Y^{t,x,z},R^{t,x,z})=(Y^{t,x,z}(s),R^{t,x,z}(s))_{s\in[0,T]}$ satisfy that, for $s\in[t,T]$,
\begin{align}\label{eq:Y}
d Y^{t,x,z}(s)=&\left(\mu \nu^{(n)}\left(\frac{a}{2 {c}}\right) -\delta Y^{t,x,z}(s)\right)ds-d R^{t,x,z}(s)+\sigma Y^{t,x,z}(s) d W(s),\quad Y^{t,x,z}(t)=x,
\end{align}
and $s\mapsto R^{t,x,z}(s)$ is a continuous and non-decreasing process that increases only on $\{s\in[t,T];~Y^{t,x,z}(s)=Z^{t,z}(s)\}$ with $R^{t,x,z}(t)=0$. Here, we recall that $Z^{t,z}=(Z^{t,z}(s))_{s\in[t,T]}$ solves \eqref{eq:E0-SDE} with $Z^{t,z}(t)=z\in\R_+$.

We further introduce the process $\tilde Y^{t,x,z}=(\tilde Y^{t,x,z}(s))_{s\in [t,T]}$ which satisfies that
\begin{align}\label{eq:tildeY}
&\tilde Y^{t,x,z}(s):=Y^{t,x,z}(s)+R^{t,x,z}(s)\\
&=x+\int_t^s\left(\mu \nu^{(n)}\left(\frac{a}{2 {c}}\right) -\delta Y^{t,x,z}(r)\right)dr+\int_t^s\sigma Y^{t,x,z}(r) d W(r).\nonumber
\end{align}
Then, using a similar argument used in the proof of  Lemma~\ref{lem:X-star-positive}, it is not difficult to conclude that $Y^{t,x,z}(s)>0$, a.s. for all $s\in[t,T]$. Consider the GBM $K^{t}=(K^{t}(s))_{s\in[t,T]}$ specified as, for $s\in[t,T]$,
\begin{align}\label{eq:GBMK}
&K^{t}(s):=\exp\left[-\left(\delta+\frac{\sigma^2}{2}\right)(s-t)+\sigma\left(W(s)-W(t)\right)\right].
\end{align}
By applying It\^{o}'s rule to $\bar{Y}^{t,x,z}(s):=(K^t(s))^{-1}Y^{t,x,z}(s)$ for $s\in [t,T]$, we have $\bar{Y}^{t,x,z}(t)=x$ and for $s\in(t,T]$,
\begin{align*}
d\bar{Y}^{t,x,z}(s)&=\mu\nu^{(n)}\left(\frac{a}{2  c}\right)\left(K^{t}(s)\right)^{-1} ds-\left(K^{t}(s)\right)^{-1} d R^{t,x,z}(s).
\end{align*}
Let $\bar{R}^{t,x,z}=(\bar{R}^{t,x,z}(s))_{s\in [t,T]}$ be a non-decreasing process which is defined by
\begin{align}\label{eq:barR}
\bar{R}^{t,x,z}(s):=\int_t^s\left(K^t(r)\right)^{-1}dR^{t,x,z}(r),~~ \forall s\in [t,T].
\end{align} 
Then, the pair of processes $(\bar{Y}^{t,x,z},\bar {R}^{t,x,z})=(\bar{Y}^{t,x,z}(s),$ $\bar {R}^{t,x,z}(s))_{s\in[t,T]}$ is the unique solution to the Skorokhod problem $(\int_t^{s}\mu\nu^{(n)}(\frac{a}{2c})(K^{t}(r))^{-1}dr,$ $(K^{t}(s))^{-1}Z^{t,z}(s))_{s\in[t,T]}$ in the sense of  Definition~\ref{def:Skorokhod}. Furthermore, we consider the process $U^{t,z}=(U^{t,z}(s))_{s\in [t,T]}$ which obeys that, for $s\in [t,T]$,
\begin{align}\label{eq:processU}
&U^{t,z}(s)=\mu\nu^{(n)}\left(\frac{a}{2 c}\right)\int_t^s\left(K^{t}(r)\right)^{-1}dr-\left(K^{t}(s)\right)^{-1}Z^{t,z}(s)\nonumber\\
&=-z+\int_t^s\left(\mu\nu^{(n)}\left(\frac{a}{2c}\right)-\Phi(Z^{t,z}(r))-(\delta+\sigma^2)Z^{t,z}(r)\right)\times \left(K^{t}(r)\right)^{-1}dr \nonumber\\
&\quad +\sigma\int_t^s\left(K^{t}(r)\right)^{-1}Z^{t,z}(r)dW(r).
\end{align}
Consequently, we have $\bar{R}^{t,x,z}(s)=(x+\sup_{r\in [t,s]}U^{t,z}(r))^{+}$ for $s\in[t,T]$. The weak solution of Robin problem \eqref{eq:Psi} is defined as follows:
\begin{definition}[Weak solution]\label{weak_solution}
Let $\mathcal{A}:=\{(x,z)\in\R_+^2;~x\leq z\}$. A function $\Psi(\cdot)\in \mathcal{W}_{\mathrm{loc}}^{1,2,1;\infty}([0,T] \times \mathcal{A})$ is called a weak solution to Robin problem \eqref{eq:Psi} if  for any test function $\phi \in {C}_0^\infty([0,T]\times \mathcal{A})$ with $\phi(0,x,z) = 0$, it holds that
\begin{align}\label{eq:weaksolution}
0&=\int_0^T\int_0^{\infty}\int_0^z \Psi(t,x,z) \Big\{ -\partial_t \phi(t,x,z) + \frac{1}{2}\sigma^2 x^2 \partial_{xx} \phi(t,x,z)\nonumber\\
&\quad+ \left( 2\sigma^2 x - \mu \nu^{(n)}\left(\frac{a}{2c}\right) + \delta x \right) \partial_x \phi(t,x,z)-\Phi(z)\partial_z\phi(t,x,z)+\left(\sigma^2+\delta-\Phi'(z)\right) \phi(t,x,z)\Big\}dxdzdt\nonumber\\
&\quad+\int_0^T\int_0^\infty \Psi(t,z,z)\bigg\{\bigg(\mu\nu^{(n)}\left(\frac{a}{2c}\right)-\delta z-\Phi(z)-\sigma^2 z\nonumber\\
&\quad-\frac{c_{\rm A}\mu^2}{4c_{\rm P}} \left(\nu^{(n)}\left(\frac{1}{2c}\right)\right)^2\bigg)\phi(t,z,z)-\frac{1}{2}\sigma^2 z^2 \partial_x \phi(t,z,z) \bigg\}dzdt\nonumber\\
&\quad+ \int_\mathcal{A} \exp\left(-\frac{c_{\rm E} \mu^2}{2c_{\rm P} \sigma^2} \left(\nu^{(n)}\left(\frac{1}{2c}\right)\right)^2 x \right)\phi(T,x,z)dxdz.
\end{align}
\end{definition}

\quad Then, we have
\begin{proposition}\label{prop:probPsi}
The probabilistic representation $\Psi(t,x,z)$ defined by \eqref{eq:Psi-R} belongs to $\mathcal{W}^{1,2,1;\infty}_{\rm loc}([0,T]\times \mathcal{A})$ and it is a weak solution to Robin problem \eqref{eq:Psi} in the sense of Definition \ref{weak_solution}. Furthermore, the optimal (feedback) control function given by \eqref{eq:star-u-R} can be characterized by
\begin{align}\label{eq:star-u-1}
&u^{*,(n)}(t,x,z)=-\frac{\mu x}{2 c_{\rm P}} \nu^{(n)}\left(\frac{1}{2 {c}}\right)\mathbb{E}\bigg[\exp\Big(-\eta^{(n)} (c_{\rm A}R^{t,x,z}(T)+ c_{\rm E} Y^{t,x,z}(T) )\Big)\\
&~\times\left(c_{\rm A}\nabla_x R^{t,x,z}(T)+c_{\rm E}\nabla_x Y^{t,x,z}(T)\right) \bigg] 
\Bigl(\mathbb{E}\Bigl[\exp\bigl(-\eta^{(n)}\left(c_{\rm A}R^{t,x,{z}}(T)+ c_{\rm E}Y^{t,x,z}(T)\right)\bigr)\Bigr] \Bigr)^{-1},\nonumber
\end{align}
where the gradient processes $\nabla_x Y^{t,x,z}=(\nabla_x Y^{t,x,z}(s))_{s\in[t,T]}$ and $\nabla_x R^{t,x,z}=(\nabla_x R^{t,x,z}(s))_{s\in[t,T]}$ are given by $\nabla_x Y^{t,x,z}(s)$ $=K^{t}(s)\boldsymbol{1}_{\{\tau^t_{x,z}> s\}}$ and $\nabla_x R^{t,x,z}(s)=K^t(\tau_{x,z}^t)\boldsymbol{1}_{\{\tau_{x,z}^t\leq s\}}$ for $s\in[t,T]$. Here, the hitting time $\tau_{x,z}^t$ is defined by
\begin{align}\label{eq:tau-u-t}
&\tau_{x,z}^t:=\inf\left\{s\in [t,T];~U^{t,z}(s)+x=0\right\}\\
&=\inf\bigg\{s\in [t,T];~z-\int_{t}^{s}\sigma Y^{t,x,z}(r) d W(r)+\int_{t}^{s} \left(\Phi(Z^{t,z}(r))-\mu \nu^{(n)}\left(\frac{a}{2 {c}}\right) +\delta Y^{t,x,z}(r)\right)dr= x\bigg\}.\nonumber
\end{align}
\end{proposition}

\quad To prove Proposition~\ref{prop:probPsi}, we need the following auxiliary result. Before presenting the auxiliary result, we define two new probability measures $\Qb_1$ and $\Qb_2$ by, given $t\in[0,T]$,
\begin{align*}
\frac{d\Qb_1}{d\Pb}|_{\F_T}=&\exp\left(-\sigma\left(W(T)-W(t)\right)-\frac{\sigma^2}{2}(T-t)\right),\\
\frac{d\Qb_2}{d\Pb}|_{\F_T}=&\exp\left(\sigma\left(W(T)-W(t)\right)-\frac{\sigma^2}{2}(T-t)\right).
\end{align*}
Define the processes $\tilde W^1=(\tilde W^1(s))_{s\in [t,T]}$ and $\tilde W^2=(\tilde W^2(s))_{s\in [t,T]}$ by $\tilde W^1(s):= W(s)+\sigma(s-t)$ and $\tilde W^2(s):=W(s)-\sigma(s-t)$ for $s\in[t,T]$. Then, by using the Girsanov's theorem, we deduce that $\tilde W^i$ is a scalar $(\Qb_i,\Fb)$-Brownian motion for $i=1,2$. Then, we have
\begin{lemma}\label{lemma:density}
Let $(t,z)\in [0,T]\times\R_+$. For fixed $s\in[t,T]$, the r.v. $\sup_{r\in [t,s]}U^{t,z}(r)$ has the density $u\mapsto p(s,u;t,z)$ under $\Pb$ and the density $u\mapsto p^{\Qb_i}(s,u;t,z)$ under $\Qb_i$ with $i=1,2$, respectively. Furthermore, the densities $u\mapsto p(s,u;t,z)$,  $u\mapsto p^{\Qb_1}(s,u;t,z)$ and $u\mapsto p^{\Qb_2}(s,u;t,z)$ are locally bounded. 
\end{lemma}
\vspace{-0.4em}
\begin{proof}
Note that $(\sigma(K^t(r))^{-1}Z^{t,z}(r))^2>0$ for all $r\in (t,T]$. Then, for $s\in(t,T]$, following the proof of Proposition 1 in \cite{Hayashi2013PA}, $r\mapsto U^{t,z}(r)$ attains its maximum on $[t,s]$ at a unique random point $\theta^{t,s}_z$, $\Qb_1$-$\as$. Furthermore, it is straightforward to verify that $U^{t,z}(s)$ is Malliavin differentiable (\citealt{Nualart}), and for any $r\in [t,T]$, the Malliavin derivative $D_rU^{t,z}(s)$ is given by
\begin{align}\label{MalliavinU}
&D_rU^{t,z}(s)=-\sigma\int_r^s\left(\mu\nu^{(n)}\left(\frac{a}{2 c}\right)-\Phi(Z^{t,z}(u))-(\delta+\sigma^2)Z^{t,z}(u)\right)(K^t(u))^{-1}du\nonumber\\
&\qquad+\sigma(K^t(r))^{-1}Z^{t,z}(r)\boldsymbol{1}_{\{r\leq s\}}-\sigma^2\int_r^s(K^t(u))^{-1}Z^{t,z}(u)dW(u)\nonumber\\
&=-\sigma\int_r^s\left(\mu\nu^{(n)}\left(\frac{a}{2c}\right)-\Phi(Z^{t,z}(u))-(\delta+2\sigma^2)Z^{t,z}(u)\right)\times(K^t(u))^{-1}du\nonumber\\
&\qquad+\sigma(K^t(r))^{-1}Z^{t,z}(r)\boldsymbol{1}_{\{r\leq s\}}-\sigma^2\int_r^s(K^t(u))^{-1}Z^{t,z}(u)d\tilde W^1(u).
\end{align}
Then, a simple calculation yields that
\begin{align*}
&\E^{\Qb_1}\Bigg[\sup_{s\in [t,T]}\left|U^{t,z}(s)\right|^2+\sup_{s\in [t,T]}\left(\int_t^T\left|D_rU^{t,z}(s)\right|^2dr\right)+\int_t^T\sup_{s\in [t,T]}\left|D_r U^{t,z}(s)\right|^2dr\Bigg]<\infty.
\end{align*}
In light of Lemma 2 in \citealt{Nakatsu2013SPL}, it follows that $\sup_{r\in [t,s]}U^{t,z}(r)\in  \mathbb{D}^{1,2}$ (the definition of $\mathbb{D}^{1,2}$ refers to Definition 1.2.1 of \citealt{Nualart}) and $D_r(\sup_{u\in [t,s]}U^{t,z}(u))=D_rU^{t,z}(\theta^{t,s}_z)$, $ds$-a.s., where $D_rU^{t,z}(\theta^{t,s}_z):=D_rU^{t,z}(u)|_{u=\theta^{t,s}_z}$. Hence, for any $s\in [t,T]$, $\Qb_1$-a.s.,
\begin{align*}
\int_t^T\left|D_r\left(\sup_{u\in [t,s]}U^{t,z}(u)\right)\right|^2dr<\infty.
\end{align*}
It then follows from Theorem 7.2.1 in \cite{Nualart} that $\sup_{r\in [t,s]}U^{t,z}(r)$ has a probability density $u\mapsto p^{\Qb_1}(s,u;t,z)$ under $\Qb_1$. The locally boundedness of $u\mapsto p^{\Qb_1}(s,u;t,z)$ follows from the proof of Theorem 2.1 in \cite{Coutin2019SPL}. The same argument can be also applied to $\Pb$ and $\Qb_2$ by using a suitable Girsanov's transformation, and hence the proof of the lemma is complete.
\end{proof}

\subsection{Proof of Proposition \ref{prop:probPsi}}

In this section, we can now prove Proposition \ref{prop:probPsi}. 

\begin{proof}[Proof of Proposition~\ref{prop:probPsi}]
We begin with the verification of the representations of gradient processes $\nabla_x Y^{t,x,z}$ and $\nabla_x R^{t,x,z}$. To do it, for any $x\in\R$, let $x_n\uparrow x$ as $n\to\infty$, it holds from \eqref{eq:barR} that, for any $s\in[t,T]$,
\begin{align*}
\frac{\bar R^{t,x_n,z}(s)-\bar R^{t,x,z}(s)}{x_n-x}=
\begin{cases}
 \displaystyle ~~~~~~~1, &  \tau_{x_n,z}^t\leq s;\\[0.6em]
 \displaystyle -\frac{\bar{R}^{t,x,z}(s)}{x_n-x}, & \tau_{x,z}^t\leq s<\tau_{x_n,z}^t;\\[0.6em]
 \displaystyle ~~~~~~~~0, & \tau_{x,z}^t>s.
\end{cases}
\end{align*}
Simple calculations yield that
\begin{align*}
&\frac{\bar Y^{t,x_n,z}(s)-\bar Y^{t,x,z}(s)}{x_n-x}=1-\boldsymbol{1}_{\{\tau_{x_n,z}^t\leq s\}}+\frac{\bar{R}^{t,x,z}(s)}{x_n-x}\boldsymbol{1}_{\{\tau_{x,z}^t\leq s<\tau_{x_n,z}^t\}}.
\end{align*}
It then holds that
\begin{align*}
&\limsup_{n\to\infty} \frac{\bar Y^{t,x_n,z}(s)-\bar Y^{t,x,z}(s)}{x_n-x}\leq \lim_{n\to\infty} \boldsymbol{1}_{\{\tau_{x_n,z}^t>s\}}=\boldsymbol{1}_{\{\tau_{x,z}^t>s\}}.
\end{align*}
On the other hand, we have by definition that
\begin{align*}
    &\frac{\bar Y^{t,x_n,z}(s)-\bar Y^{t,x,z}(s)}{x_n-x}=\boldsymbol{1}_{\{\tau_{x,z}^t>s\}}+\left(1+\frac{\bar{R}^{t,x,z}(s)}{x_n-x}\right)\boldsymbol{1}_{\{\tau_{x,z}^t\leq s<\tau_{x_n,z}^t\}}\\
    &=\boldsymbol{1}_{\{\tau_{x,z}^t>s\}}+\frac{x_n+\sup_{r\in t,s]}U^{t,z}(r)}{x_n-x}\boldsymbol{1}_{\{\tau_{x,z}^t\leq s<\tau_{x_n,z}^t\}}\geq \boldsymbol{1}_{\{\tau_{x,z}^t>s\}}.
\end{align*}
Here, the last inequality holds from  the fact $x_n+\sup_{r\in t,s]}U^{t,z}(r)\leq 0$ when $s<\tau_{x_n,z}^t$. 
Then we deduce that\begin{align*}
    \lim_{n\to\infty}\frac{\bar Y^{t,x_n,z}(s)-\bar Y^{t,x,z}(s)}{x_n-x}=\boldsymbol{1}_{\{\tau_{x,z}^t>s\}}.
\end{align*}
Consequently, we have, for all $s\in[t,T]$,
\begin{align*}
&\lim_{n\to\infty} \frac{Y^{t,x_n,z}(s)-Y^{t,x,z}(s)}{x_n-x}=\lim_{n\to\infty}\frac{(\bar Y^{t,x_n,z}(s)-\bar Y^{t,x,z}(s))K^{t}(s)}{x_n-x}=K^{t}(s)\boldsymbol{1}_{\{\tau_{x,z}^t>s\}}.
\end{align*}
The same argument can also be applied to the case where $x_n\downarrow x$. Furthermore, by the chain rule and \eqref{eq:GBMK}, for any $s\in[t,T]$,
\begin{align*}
\nabla_x\tilde Y^{t,x,z}(s)&=1-\int_t^s\delta\nabla_x Y^{t,x,z}(r)dr+\int_t^s\sigma\nabla_xY^{t,x,z}(r)dW(r)\\
&=1-\int_t^{\tau_{x,z}^t\wedge s}\delta K^{t}(r)d r+\int_t^{\tau_{x,z}^{t}\wedge s}\sigma K^t(r)dW(r)=K^{t}(\tau_{x,z}^{t}\wedge s).
\end{align*}
Then $\nabla_x R^{t,x,z}(s)=\nabla_x\tilde{Y}^{t,x,z}(s)-\nabla_xY^{t,x,z}(s)=K^{t}(\tau_{x,z}^{t}\wedge s)-K^{t}(s)\boldsymbol{1}_{\{\tau_{x,z}^t>s\}}=K^{t}(\tau_{x,z}^{t})\boldsymbol{1}_{\{\tau_{x,z}^t\leq s\}}$. On the other hand, using DCT and the chain rule, it follows from \eqref{eq:Psi-R} that
\begin{align}\label{eq:paxPsi}
&\pa_x\Psi(t,x,z)\nonumber\\
&=-\eta^{(n)}\E\left[\exp\left(-\eta^{(n)}\left(c_{\rm A} R^{t,x,z}(T)+c_{\rm E}Y^{t,x,z}(T)\right)\right)\left(c_{\rm A}\nabla_x R^{t,x,z}(T)+c_{\rm E}\nabla_x Y^{t,x,z}(T)\right)\right]\nonumber\\
&=:I_1(t,x,z)+I_2(t,x,z),
\end{align}
where the terms $I_i(t,x,z)$ for $i=1,2$ are given by
\begin{align*}
I_1(t,x,z)&:= -\eta^{(n)}c_{\rm A}\E\left[\exp\left(-\eta^{(n)}\left(c_{\rm A} R^{t,x,z}(T)+c_{\rm E}Y^{t,x,z}(T)\right)\right)\nabla_x R^{t,x,z}(T)\right],\nonumber\\
I_2(t,x,z)&:=-\eta^{(n)}c_{\rm E}\E\left[\exp\left(-\eta^{(n)}\left(c_{\rm A} R^{t,x,z}(T)+c_{\rm E}Y^{t,x,z}(T)\right)\right)\nabla_x Y^{t,x,z}(T)\right]. 
\end{align*}

\quad We next show that $x\mapsto\pa_x\Psi(t,x,z)$ is locally Lipschitz on $\R_+$. Fix $x_1,x_2\in\R_+$  with $x_1<x_2$. Then, inserting $\nabla_xR^{t,x,z}(s)=K^{t}(\tau_{x,z}^{t})\boldsymbol{1}_{\{\tau_{x,z}^t\leq s\}}$ into \eqref{eq:paxPsi}, we have
\begin{align}\label{eq:I_1decomp}
&|I_1(t,x_1,z)-I_1(t,x_2,z)|\nonumber\\
&\leq\eta^{(n)}c_{\rm A}\bigg\{\E\Big[\exp\left(-\eta^{(n)}\left(c_{\rm A} R^{t,x_1,z}(T)+c_{\rm E}Y^{t,x_1,z}(T)\right)\right) \left|K^{t}(\tau_{x_1,z}^t)\boldsymbol{1}_{\{\tau_{x_1,z}^t\leq T\}}-K^{t}(\tau_{x_2,z}^t)\boldsymbol{1}_{\{\tau_{x_2,z}^t\leq T\}}\right|\Big]\nonumber\\
&+\E\Big[\Big|\exp\left(-\eta^{(n)}\left(c_{\rm A} R^{t,x_2,z}(T)+c_{\rm E}Y^{t,x_2,z}(T)\right)\right)\nonumber\\
&\qquad-\exp\left(-\eta^{(n)}\left(c_{\rm A} R^{t,x_1,z}(T)+c_{\rm E}Y^{t,x_1,z}(T)\right)\right)\Big| \times K^t(\tau_{x_2,z}^t)\boldsymbol{1}_{\{\tau_{x_2,z}^t\leq T\}}\Big]\bigg\}.
\end{align}
Since $x\mapsto R^{t,x,z}(T)$ and $x\mapsto\tilde Y^{t,x,z}(T)$ are differentiable, there exists a constant $C_F>0$ only dependent of $(T,z_0,m,M)$ such that the 2nd term in \eqref{eq:I_1decomp} is dominated by $C_F|x_1-x_2|$ for any $x_1,x_2\in F$ with $F\subset \R_+=(0,\infty)$  being a compact subset. For the 1st term in \eqref{eq:I_1decomp}, it holds that
\begin{align}\label{eq:I11_decomp}
&\E\Bigl[\exp\left(-\eta^{(n)}\left(c_{\rm A} R^{t,x_1,z}(T)+c_{\rm E}Y^{t,x_1,z}(T)\right)\right) \left|K^t(\tau_{x_1,z}^t)\boldsymbol{1}_{\{\tau_{x_1,z}^t\leq T\}}-K^t(\tau_{x_2,z}^t)\boldsymbol{1}_{\{\tau_{x_2,z}^t\leq T\}}\right|\Bigr]\nonumber\\
&\leq\E\left[\left|K^t(\tau_{x_1,z}^t)\boldsymbol{1}_{\{\tau_{x_1,z}^t\leq T\}}-K^t(\tau_{x_2,z}^t)\boldsymbol{1}_{\{\tau_{x_2,z}^t\leq T\}}\right|\right]\nonumber\\
&=\E\left[\left|\frac{Z^{t,z}(\tau_{x_1,z}^t)\boldsymbol{1}_{\{\tau_{x_1,z}^t\leq T\}}}{x_1+\mu\nu^{(n)}\left(\frac{a}{2 c}\right)\int_t^{\tau_{x_1,z}^t}(K^{t}(r))^{-1}dr}-\frac{Z^{t,z}(\tau_{x_2,z}^t)\boldsymbol{1}_{\{\tau_{x_2,z}^t\leq T\}}}{x_2+\mu\nu^{(n)}\left(\frac{a}{2c}\right)\int_t^{\tau_{x_2,z}^t}(K^{t}(r))^{-1}dr}\right|\right]\nonumber\\
&\leq\E\Bigg[\left(x_2+\mu\nu^{(n)}\left(\frac{a}{2c}\right)\int_t^{\tau_{x_2,z}^t}(K^{t}(r))^{-1}dr\right)\left|Z^{t,z}(\tau_{x_1,z}^t)\boldsymbol{1}_{\{\tau_{x_1,z}^t\leq T\}}-Z^{t,z}(\tau_{x_2,z}^t)\boldsymbol{1}_{\{\tau_{x_2,z}^t\leq T\}}\right|\nonumber\\
&\qquad\times\left(x_1+\mu\nu^{(n)}\left(\frac{a}{2c}\right)\int_t^{\tau_{x_1,z}^t}(K^{t}(r))^{-1}dr\right)^{-1}\left(x_2+\mu\nu^{(n)}\left(\frac{a}{2c}\right)\int_t^{\tau_{x_2,z}^t}(K^{t}(r))^{-1}dr\right)^{-1}\Bigg]\nonumber\\
&\quad+\E\Bigg[\left|x_1-x_2+\mu\nu^{(n)}\left(\frac{a}{2 c}\right)\int_{\tau_{x_2,z}^t}^{\tau_{x_1,z}^t}(K^{t}(r))^{-1}dr\right|  Z^{t,z}(\tau_{x_2,z}^t)\boldsymbol{1}_{\{\tau_{x_2,z}^t\leq T\}}\nonumber\\
&\qquad\times \left(x_1+\mu\nu^{(n)}\left(\frac{a}{2 c}\right)\int_t^{\tau_{x_1,z}^t}(K^t(r))^{-1}dr\right)^{-1} \left(x_2+\mu\nu^{(n)}\left(\frac{a}{2c}\right)\int_t^{\tau_{x_2,z}^t}(K^t(r))^{-1}dr\right)^{-1}\Bigg]\nonumber\\
&\quad=:J_1+J_2.
\end{align}
Here, for the 1st equality in \eqref{eq:I11_decomp}, we used the fact that, for $(t,x,z)\in [0,T]\times \mathcal{A}$, $x+\mu\nu^{(n)}\left(\frac{a}{2c}\right)\int_t^{\tau_{x,z}^t}(K^t(r))^{-1}dr=(K^t(\tau_{x,z}^t))^{-1}Z^{t,z}(\tau_{x,z}^t)$ on $\{\tau_{x,z}^t\leq T\}$, a.s..

\quad For the term $J_1$, since $K^t(r)>0$ for all $r\in[0,T]$ and the fact that  $\sup_{r\in[t,T]}|Z^{t,z}(r)|\leq C$ for some constant $C=C(T,z,L)$ which only depends on $(T,z,L)$, we deduce that
\begin{align*}
&J_1\leq x_1^{-1}\E\left[\left|Z^{t,z}(\tau_{x_1,z}^t)\boldsymbol{1}_{\{\tau_{x_1,z}^t\leq T\}}-Z^{t,z}(\tau_{x_2,z}^t)\boldsymbol{1}_{\{\tau_{x_2,z}^t\leq T\}}\right|\right]\\
&\leq x_1^{-1}\Bigl\{2C\Pb\left(\tau_{x_2,z}^t\leq T<\tau_{x_1,z}^t\right) +\E\left[\left|Z^{t,z}(\tau_{x_1,z}^t)-Z^{t,z}(\tau_{x_2,z}^t)\right|\boldsymbol{1}_{\{\tau_{x_1,z}^t\leq T\}}\right]\Bigr\}\\
&\leq x_1^{-1}\Biggl\{2C\Pb\left(-x_2\leq \sup_{s\in [t,T]}U^{t,z}(s)\leq -x_1\right)+\E\left[\int_{\tau_{x_2,z}^t}^{\tau_{x_1,z}^t}\Phi(Z^{t,z}(r))d r\boldsymbol{1}_{\{\tau_{x_1,z}^t\leq T\}}\right]\Biggr\}\\
&\leq Cx_1^{-1}\Biggl\{2\int_{-x_2}^{-x_1}p(T,u;t,z)d u+C\E\left[\int_0^T\boldsymbol{1}_{\{\tau_{x_2,z}^t\leq r\leq\tau_{x_1,z}^t\leq T\}}dr\right]\Biggr\}\\
&\leq Cx_1^{-1}\Bigg\{2\int_{-x_2}^{-x_1}p(T,u;t,z)d u+C\int_0^T\int_{-x_2}^{-x_1}p(r,u;t,z)dudr\Bigg\}.
\end{align*}
Hence, $J_1$ is can be bounded by $C_F|x_1-x_2|$ if $x_1,x_2\in F$ with $F\subset \R_+$ being an arbitrary compact set due to Lemma~\ref{lemma:density}. $C_F$ only depends on $(T,z,L)$ and the choice of the compact set $F$. On the other hand, it holds from Fubini's theorem that
\begin{align*}
&J_2\leq Cx_1^{-2}\Biggl[|x_1-x_2| +\mu\nu^{(n)}\left(\frac{a}{2 c}\right)\E\left[\int_{\tau_{x_2}^t}^{\tau_{x_1,z}^t}(K^t(r))^{-1}dr \right]\Biggr]\\
&= Cx_1^{-2}\Biggl[|x_1-x_2|+\mu\nu^{(n)}\left(\frac{a}{2 c}\right)\int_0^T\E\left[(K^t(r))^{-1}\boldsymbol{1}_{\{\tau_{x_2,z}^t\leq r\leq\tau_{x_1,z}^t\leq T\}}\right]dr\Biggr]\\
&\leq Cx_1^{-2}\Bigg[|x_1-x_2|+\mu\nu^{(n)}\left(\frac{a}{2 c}\right)\int_0^T\int_{-x_2}^{-x_1}p^{\Qb_1}(r,u;t,z)d udr\Bigg].
\end{align*}
Repeating the argument for $J_1$, we can establish the locally Lipschitz property for $x\mapsto I_1(t,x,z)$. For $I_2(t,x,z)$, it readily follows from  representation $\nabla_x Y^{t,x,z}(s)=K^{t}(s)\boldsymbol{1}_{\{\tau^t_{x,z}> s\}}$ that
\begin{align}\label{eq:I_2decomp}
& |I_2(t,x_1,z)-I_2(t,x_2,z)|\leq\eta^{(n)}c_{\rm E}\nonumber\\
&\times\Biggl\{\E\Biggl[\exp\left(-\eta^{(n)}\left(c_{\rm A} R^{t,x_1,z}(T)+c_{\rm E}Y^{t,x_1,z}(T)\right)\right)K^t(T)\times\left|\boldsymbol{1}_{\{\tau_{x_1,z}^t>T\}}-\boldsymbol{1}_{\{\tau_{x_2,z}^t>T\}}\right|\Biggr]\nonumber\\
&\quad+\E\Big[\Big|\exp\left(-\eta^{(n)}\left(c_{\rm A} R^{t,x_2,z}(T)+c_{\rm E}Y^{t,x_2,z}(T)\right)\right)\nonumber\\
&\qquad-\exp\left(-\eta^{(n)}\left(c_{\rm A} R^{t,x_1,z}(T)+c_{\rm E}Y^{t,x_1,z}(T)\right)\right)\Big|\times K^t(T)\boldsymbol{1}_{\{\tau_{x_2,z}^t>T\}}\Big]\Biggr\}.
\end{align}
Similarly, using the argument as for $x\mapsto I_1(t,x,z)$, we have that the 2nd term in \eqref{eq:I_2decomp} is locally Lipschitz in $x\in\R_+$. On the other hand, by construction of the probability measure $\Qb_2$, it holds that
\begin{align}\label{eq:pQTintegral}
&\E\left[\exp\left(-\eta^{(n)}\left(c_{\rm A} R^{t,x_1,z}(T)+c_{\rm E}Y^{t,x_1,z}(T)\right)\right)K^t(T)\times\left(\boldsymbol{1}_{\{\tau_{x_1,z}^t>T\}}-\boldsymbol{1}_{\{\tau_{x_2,z}^t>T\}}\right)\right]\nonumber\\
&\leq \Qb_2\left(\tau_{x_2,z}^t\leq T<\tau_{x_1,z}^t\right)\leq\Qb_2\left(-x_2\leq \sup_{s\in [t,T]}U^{t,z}(s)\leq -x_1\right)\nonumber\\
&=\int_{-x_2}^{-x_1}p^{\Qb_2}(T,u;t,z)du.
\end{align}
By using Lemma~\ref{lemma:density}, there exists a constant $C_F>0$ such that \eqref{eq:pQTintegral} can be bounded by $C_F|x_1-x_2|$ for every $x_1,x_2\in F$ with $F\subset\R_+$ being compact. This yields that $x\mapsto I_2(t,x,z)$ is locally Lipschitz, and hence $x\to\pa_x\Psi(t,x,z)$ is also locally Lipschitz. Hence, $x\mapsto\pa_{xx}\Psi(t,x,z)$ exists and belongs to $L^{\infty}_{\rm loc}(\R_+)$.

\quad We next prove that $z\mapsto\Psi(t,x,z)$ is also Lipschitz. To do it, we first recall that, for $s\in[t,T]$,
\begin{align*}
\bar R^{t,x,z}(s)=&\sup_{r\in [t,s]}\left[x+\mu\nu^{(n)}\left(\frac{a}{2 c}\right)\int_t^r(K^t(l))^{-1}dl-(K^t(r))^{-1}Z^{t,z}(r)\right]\vee 0.
\end{align*}
By utilizing the fact that $\pa_z Z^{t,z}(s)=\exp\left(\int_t^s\Phi'(Z^{t,z}(r))dr\right)\leq e^{ L T}$ for all $s\in[t,T]$, where  we recall that $L>0$ is the Lipschitz constant of $z\mapsto\Phi(z)$, we deduce that,  for any $s\in [t,T]$ and $z_1,z_2>0$,
\begin{align*}
&\left|\bar Y^{t,x,z_1}(s)-\bar Y^{t,x,z_2}(s)\right|=\left|\bar R^{t,x,z_1}(s)-\bar R^{t,x,z_2}(s)\right|\leq e^{ LT}|z_1-z_2|\sup_{r\in [t,s]}(K^t(r))^{-1}.
\end{align*} 
This implies that $|Y^{t,x,z_1}(s)-Y^{t,x,z_2}(s)|\leq e^{LT}|z_1-z_2|K^t(s)\sup_{r\in [t,s]}(K^t(r))^{-1}$ for all $s\in[t,T]$. From Burkholder-Davis-Gundy (BDG) inequality and Gronwall's inequality, it follows that
\begin{align*}
&\E\left[\left|Y^{t,x,z_1}(T)-Y^{t,x,z_2}(T)\right|^2+\left|R^{t,x,z_1}(T)-R^{t,x,z_2}(T)\right|^2\right]\leq C|z_1-z_2|^2,
\end{align*}
for some constant $C=C(T,L)>0$ depending on $(T,L)$. Consequently, we arrive at
\begin{align}\label{eq:PsiLipz}
&\left|\Psi(t,x,z_1)-\Psi(t,x,z_2)\right|\nonumber\\
&\leq \E\left[e^{-\eta^{(n)}c_{\rm A}R^{t,x,z_1}(T)}\left|e^{-\eta^{(n)}c_{\rm E}Y^{t,x,z_1}(T)}-e^{-\eta^{(n)}c_{\rm E}Y^{t,x,z_2}(T)}\right|\right]\nonumber\\
&\quad+\E\left[e^{-\eta^{(n)}c_{\rm E}Y^{t,x,z_2}(T)}\left|e^{-\eta^{(n)}c_{\rm A}R^{t,x,z_1}(T)}-e^{-\eta^{(n)}c_{\rm A}R^{t,x,z_2}(T)}\right|\right]\nonumber\\
&\leq \eta^{(n)}\left\{c_{\rm A}\E\left[|R^{t,x,z_1}(T)-R^{t,x,z_2}(T)|\right]+c_{\rm E}\E\left[|Y^{t,x,z_1}(T)-Y^{t,x,z_2}(T)|\right]\right\}\nonumber\\
&\leq \sqrt{C}\eta^{(n)}\max\{c_{\rm A},c_{\rm E}\}|z_1-z_2|.
\end{align}

Lastly, we show that $t\mapsto\Psi(t,x,z)$ is locally Lipschitz continuous on $[0,T)$. To this purpose, we first verify that $t\mapsto p(T,u;t,z)$ is locally Lipschitz continuous on $[0,T)$. Define a new probability measure $\Qb$ by $\frac{d\Qb}{d\Pb}|_{\F_T}=\mathcal{E}(E^t)$, 
where the exponential martingale $\mathcal{E}(E^t)$ is defined by $\mathcal{E}(E^t)(\cdot)=\exp\left(E^t(\cdot)-\frac12\langle E^t,E^t\rangle_{\cdot}\right)$ and the martingale process $E^t=(E^t(s))_{s\in [t,T]}$ is given by
\begin{align*}
&E^{t}(s):=-\int_0^s\frac{\mu\nu^{(n)}\left(\frac{a}{2c}\right)-\Phi(Z^{t,z}(r))-(\delta+\sigma^2)Z^{t,z}(r)}{\sigma Z^{t,z}(r)}dW(r).
\end{align*}
Then, we have $\tilde W(s):=W(s)+\langle E^t,W\rangle_s$ for $s\in[t,T]$ is a $(\Qb,\Fb)$-Brownian motion. Furthermore, let $V^{t,z}(s):=-z+\int_t^s\sigma Z^{t,z}(r)(K^t(r))^{-1}dW(r)$ for $s\in[t,T]$. As a result, we have $\Qb\circ (V^{t,z})^{-1}=\Pb\circ (U^{t,z})^{-1}$. Introduce the time-change $s\mapsto\Pi^t(s):=t+\int_t^s(\sigma Z^{t,z}(r)(K^t(r))^{-1})^2dr$ which is continuous and strictly increasing on $[t,T]$. Denote by $s\mapsto\kappa^{t}(s)$ the inverse function of $s\mapsto\Pi^{t}(s)$, and define a new filtration by $\G_s:=\F_{\kappa^{t}(s)}$ for $s\in[t,T]$. Hence, $V^{t,z}(\kappa^{t}(s))=-z+\int_t^{\kappa^{t}(s)}\sigma Z^{t,r}(r)(K^t(r))^{-1}dW(r)$ for $s\in[t,T]$ is a $\mathbb{G}=(\G_s)_{s\in[t,T]}$-Brownian motion under $\Pb$. Moreover, $s\mapsto\Pi^{t}(s)$ is strictly increasing and differentiable with respect to $s$ and so is $s\mapsto\kappa^{t}(s)$. For every $m\in\R_+$, $\Pb (\sup_{r\in [t,s]}V^{t,z}(r)\leq m )=\Pb (\sup_{r\in [t,\Pi^t(s)]}W(r)\leq m )$ is locally Lipschitz in $t\in [0,T)$. Consequently, $\Pb (\sup_{r\in [t,s]}U^{t,z}(r)\leq m )=\Qb(\sup_{r\in [t,s]}V^{t,z}(r)\leq m)$ is locally Lipschitz  by using local Lipschitz continuity of $E^t$. As a result, $p(s,u;t,z)$ is locally Lipschitz in $t\in [0,T)$. 
    
\quad We now fix $h>0$ and $t\in [0,T)$ such that $t+h< T$. Since we have established that $\Psi(t,\cdot,\cdot)\in \mathcal{W}^{2,1;\infty}(\mathcal{A})$ for each $t\in [0,T)$, we can apply It\^{o}'s formula to $\Psi(t+h,Y^{t,x,z}(t+h),Z^{t,z}(t+h))-\Psi(t+h,x,z)$ (c.f., Theorem 1 in Chapter 2 Section 10 of \citealt{Krylov}) to derive that
\begin{align}\label{ItoPsi}
&\Psi(t+h,Y^{t,x,z}(t+h),Z^{t,z}(t+h))-\Psi(t+h,x,z)=\int_t^{t+h}\pa_x\Psi(t+h,Y^{t,x,z}(s),Z^{t,z}(s))dY^{t,x,z}(s)\nonumber\\
&+\frac12\int_t^{t+h}\pa_{xx}\Psi(t+h,Y^{t,x,z}(s),Z^{t,z}(s))d\langle Y^{t,x,z},Y^{t,x,z}\rangle_s+\int_t^{t+h}\pa_z\Psi(t+h,Y^{t,x,z}(s),Z^{t,z}(s))dZ^{t,z}(s).
\end{align}
Inserting the Markov property $\Psi(t,x,z)=\E [ \Psi(t+h,$ $Y^{t,x,z}(t+h),Z^{t,z}(t+h)) ]$ into \eqref{ItoPsi}, we deduce that
\begin{align}\label{eq:Psitdif}
&\Psi(t,x,z)-\Psi(t+h,x,z)\nonumber\\
&=\E\Bigg[\int_t^{t+h}\pa_x\Psi(t+h,Y^{t,x,z}(s),Z^{t,z}(s))\times\left(\mu\nu^{(n})\left(\frac{1}{2 c}\right)-\delta Y^{t,x,z}(s)\right)ds\Bigg]\nonumber\\
&\quad+\frac{1}{2}\E\left[\int_t^{t+h}\pa_{xx}\Psi(t+h,Y^{t,x,z}(s),Z^{t,z}(s))(\sigma Y^{t,x,z}(s))^2ds\right]\nonumber\\
&\quad+\E\left[\int_t^{t+h}\pa_z\Psi(t+h,Y^{t,x,z}(s),Z^{t,z}(s))\Phi(Z^{t,z}(s))ds\right]\nonumber\\
&\quad-\E\left[\int_t^{t+h}\pa_x\Psi(t+h,Y^{t,x,z}(s),Z^{t,z}(s))d R^{t,x,z}(s)\right]\nonumber\\
&=L_1+L_2+L_3-L_4.
\end{align}
It is straightforward to verify that $L_1+L_2+L_3$ can be bounded by $Ch$ for some constant $C>0$ depending only on $(T,z,L,m,M)$. For the term $L_4$, note that
\begin{align*}
\frac{L_4}{h}&\leq\frac{1}{h}\E\Bigg[\Bigg|\int_t^{t+h} \bigg(\pa_x\Psi(t+h,Y^{t,x,z}(s),Z^{t,z}(s))-\pa_x\Psi(t+h,x,z)\bigg) d R^{t,x,z}(s)\Bigg|\Bigg]\\
&\quad+\frac{\pa_x\Psi(t+h,x,z)}{h}\E\left[R^{t,x,z}(t+h)-R^{t,x,z}(t)\right].
\end{align*}
The 1st term of R.H.S. of the above display tends to $0$ as $h\to 0$ due to Lebesgue differentiation theorem, and the 2nd term is locally bounded by the local Lipschitz continuity of $p(s,u;t,z)$ w.r.t. $t\in [0,T)$ and \eqref{eq:barR}. Combining all the above arguments, we verified $\Psi\in\mathcal{W}^{1,2,1;\infty}([0,T]\times \mathcal{A})$. Lastly, by applying the generalized It\^{o} formula for functions in Soblev space (c.f., Theorem~1 in Chapter 2 Section 10 of \citealt{Krylov}) to $\Psi(t,Y^{t,x,z}(s),Z^{t,x,z}(s))$  $s\in[t,T]$, we conclude that $\Psi(t,x,z)$ is indeed the weak solution to \eqref{eq:Psi} in the sense of Definition \ref{weak_solution}, where the boundary condition and the terminal condition can be directly verified. Thus, the proof of the proposition is complete. 
\end{proof}

\quad Then, we have the following verification result on the leader's optimal control problem:
\begin{corollary}\label{prop:verification}
Let $u^{*,(n)}(t,x,z)$ be given in Proposition \ref{prop:probPsi}. Define $u^{*,(n)}(t):=u^{*,(n)}(t,X^{*,(n)},Z(t))$ for $t\in[0,T]$, where $(X^{*,(n)},A^{*,(n)})=(X^{*,(n)}(t),A^{*,(n)}(t))_{t\in[0,T]}$ satisfies \eqref{eq:n-opt-X-F} with production strategy $\boldsymbol{q}^{*,(n)}(u)$ given by \eqref{eq:n-opt-q} in which $u$ is replaced by $u^{*,(n)}$. 
Then, $u^{*,(n)}=(u^{*,(n)}(t))_{t\in[0,T]}$ is an leader's optimal strategy and the corresponding value function is given by $V(t,x,z)=-\frac{1}{\eta^{(n)}}\log\Psi(t,x,z)$ with $\Psi$ being the probabilistic representation \eqref{eq:Psi-R}.
\end{corollary}
\vspace{-0.4cm}
\begin{proof}
We have, for some constant$C_0=C_0(T,m,M)>0$,
\begin{align*}
&\E\left[R^{t,x,z}(T)\right]=\E\left[\int_t^TK^t(s)d\bar{R}^{t,x,z}(s)\right]\leq \E\left[\left(\sup_{s\in [t,T]}K^t(s)\right)\bar{R}^{t,x,z}(T)\right]\\
&\leq \left(\E\left[\sup_{s\in [t,T]}|K^t(s)|^2\right]\right)^{\frac12}\left(\E\left[|\bar R^{t,x,z}(T)|^2\right]\right)^{\frac12}\leq C_0\left(\E\left[\left| x+\sup_{s\in [t,T]}U^{t,z}(s)\right|^2\right]\right)^{\frac12}.
\end{align*}
It follows from Jensen's inequality that
\begin{align*}
    &\E\left[\exp\left(-\eta^{(n)}\left(c_{\rm A} R^{t,x,z}(T)+c_{\rm E}Y^{t,x,z}(T)\right)\right)\right]\\
    \geq &\exp\left(-\eta^{(n)}\E\left[c_{\rm A} R^{t,x,z}(T)+c_{\rm E}Y^{t,x,z}(T)\right]\right)\\
    \geq&\exp\Big(-\eta^{(n)}\Big(c_{\rm A}C_0\Big(\E\big[\big|(z+\sup_{s\in [t,T]}U^{t,z}(s)\big|^2\big]\Big)^{\frac12}+c_{\rm E}Z^{t,z}(T)\Big)\Big)=:C_3.
\end{align*}
Here, we have used the fact $Y^{t,x,z}(T)\leq Z^{t,z}(T)$ and $x<z$. Hence, it is straightforward to verify that $C_3$ is a constant only depending on $(m,M,L,T,z_0)$. 
Moreover, note that $\sup_{s\in [t,T]}|\nabla Y^{t,x,z}(s)|\vee|\nabla_x R^{t,x,z}(s)|$ $\leq \sup_{s\in [t,T]}K^{t}(s)$, a.s. by using Proposition \ref{prop:probPsi}. Then, from the representation of $\nabla_x Y^{t,x,z}$ and $\nabla_x\tilde Y^{t,x,z}$ given in Proposition \ref{prop:probPsi}, there is a constant $\kappa>0$ depending on $(T,m,M)$ such that $\E[c_{\rm A}\nabla_x R^{t,x,z}(T)$ $+c_{\rm E}\nabla_xY^{t,x,z}(T)]\leq \kappa$. Recall the optimal (feedback) control function $u^{*,(n)}(t,x,z)$ given by \eqref{eq:star-u-1}. Then, it follows from H\"{o}lder's inequality that $|u^{*,(n)}(t,x,z)|\leq \frac{C_3\kappa\mu\eta^{(n)}}{2c_{\rm P}}\nu^{(n)} (\frac{1}{2c} ) x$ for all $(t,x,z)\in[0,T]\times \mathcal{A}$, where we exploited the fact $Y^{t,x,z}(T)\geq 0$, $R^{t,x}(T)\geq 0$ and $Y^{t,x,z}(T)\leq Z^{t,z}(T)$. Thanks to Proposition~\ref{prop:probPsi}, we have $x\mapsto u^{*,(n)}(t,x,z)$ is locally Lipschitz, and hence the SDE \eqref{eq:n-opt-X-F} controlled by $u^{*,(n)}=(u^{*,(n)}(t,X^{*,u,(n)},Z(t)))_{t\in[0,T]}$ has a unique solution. Moreover,  there exists a constant $\alpha:=\alpha( m,M,L,T,z_0)$ such that $\frac{C_3\kappa\mu\eta^{(n)}}{2c_{\rm P}}\nu^{(n)}\left(\frac{1}{2c}\right)\leq \alpha$. Combined with the fact
 $\sup_{n\geq1} \E[\sup_{t\in[0,T]}|X^{*,u,(n)}(t)|^2]\leq Z(T)^2$, 
 we have
\begin{align}\label{uniformL2_u}
&\int_0^T \E\left[\left|u^{*,(n)}(t)\right|^2\right]dt=\int_0^T\E\left[\left|u^{*,(n)}(t,X^{*,u,(n)}(t),Z(t))\right|^2\right]dt\leq \alpha Z(T)^2.
\end{align}
Consequently, $u^{*,(n)}=(u^{*,(n)}(t,X^{*,u,(n)},Z(t)))_{t\in[0,T]}\in\mathbb{U}_0$, i.e., it is an admissible control. On the other hand, for any admissible control $u=(u(t))_{t\in [0,T]}\in\mathbb{U}_0$, by applying the generalized It\^{o}'s formula for functions in Soblev space (c.f., Theorem~1 in Chapter 2 Section 10 of \citealt{Krylov}) to $V(t,X^{u}(t),Z(t))$ with $V(t,x,z)-\frac{1}{\eta^{(n)}}\log\Psi(t,x,z)$ and the state process $X^u=(X^u(t))_{t\in [0,T]}$ satisfying \eqref{eq:star-X-u-1}, we can finally derive that $J_0^{(n)}(u^{*,(n)};\boldsymbol{q}^{*,(n)}(u^{*,(n)}))\leq  J_0^{(n)}(u;\boldsymbol{q}^{*,(n)}(u))$ for any $u\in\mathbb{U}_0$. Thus, the proof of the corollary is complete.
\end{proof}

\subsection{Approximate Stackelberg Equilibrium}\label{sec:ASE}

This subsection shows that if the leader announces $u^{*,(n)}$ obtained in Section~\ref{sec:leader} to the $n$ followers, the set of optimal strategies for the leader and the followers constitutes an $(\epsilon,0)$-Stackelberg equilibrium.

\begin{theorem}
The leader's strategy $u^{*,(n)}\in\mathbb{U}_0$ given by \eqref{eq:star-u-1} and the strategy set of $n$ followers $(q_1^{*,(n)}(u^{*,(n)}),\dots,$ $q_n^{*,(n)}(u^{*,(n)}) )\in\mathbb{U}^n$ given by~\eqref{eq:n-opt-q} establish an $(\epsilon,0)$-Stackelberg equilibrium, where $\epsilon=\epsilon(n)>0$ satisfying $\lim_{n\to\infty}\epsilon(n)=0$ is provided in Theorem~\ref{thm:eps-Nash}.
\end{theorem}
\vspace{-0.4em}
\begin{proof}
It follows from Theorem~\ref{thm:eps-Nash} that there exists a constant $\epsilon=\epsilon(n)$ satisfying $\lim_{n\to\infty}\epsilon(n)=0$ such that $(q_1^{*,(n)}(u),\dots,q_n^{*,(n)}(u) )$ given by \eqref{eq:n-opt-q} constitutes an $\epsilon$-Nash equilibrium under any $u=(u(t))_{t\in [0,T]}\in\mathbb{U}_0$, 
i.e., $ J_{i}^{(n)}(\boldsymbol{q}^{*,(n)}(u) )\geq\sup_{q_{i}\in\mathbb{U}_i} J_{i}^{(n)}(q_{i},$ $\boldsymbol{q}^{*,(n),-i}(u) )-\epsilon$ for all $i=1,\ldots,n$. On the other hand, for the leader's strategy $u^{*,(n)}\in\mathbb{U}_0$ given by \eqref{eq:star-u-1}, by the verification result documented in Corollary \ref{prop:verification}, under the vector of strategies $(q_1^{*,(n)}(u^{*,(n)}),\dots,q_n^{*,(n)}(u^{*,(n)}))$ given by \eqref{eq:n-opt-q}, we have $J_0^{(n)}(u^{*,(n)};\boldsymbol{q}^{*,(n)}(u^{*,(n)}) )\leq \inf_{u\in\mathbb{U}_0} J_0^{(n)}(u;\boldsymbol{q}^{*,(n)}(u))$. Thus, we established the conditions (i) and (ii) of Definition \ref{def:SE} with $(\epsilon_1,\epsilon_2)=(\epsilon(n),0)$. Hence, the proof of the theorem is complete.
\end{proof}

\section{Numerical Analysis}\label{sec:number}

This section performs numerical analysis on approximate Stackelberg equilibria  
with different regional quantities. We then quantify the sensitivity of equilibrium solutions with respect to perturbations of model parameters. Consider the cap function $\Phi:\R\to\R$ in \eqref{eq:E0-SDE} which takes the reverting linear form $ \Phi(z)=A_Z-B_Z z$ with $A_Z,B_Z>0$. Then, the emission cap process $Z=(Z(t))_{t\in[0,T]}$ has the closed-form representation given by, for $z_0>0$,
\begin{align*}
Z(t)=\frac{A_Z}{B_Z}+\left(z_0-\frac{A_Z}{B_Z}\right)e^{-B_Zt},\quad \forall t\in[0,T].    
\end{align*}
Here, the parameter of cap function $A_Z$ is the steady-state cap level coefficient of the emission cap process, which sets the baseline level toward which the emission cap stabilizes in the long run. The parameter $B_Z$ is the adjustment speed coefficient of the emission cap process, which measures the intensity with which the policy maker corrects deviations of the current emission cap from the long‑term target. 

\begin{figure}[htbp]
	\centering
	\subfigure[Average emission process simulation with different values of $A_Z$.]{
		\includegraphics[width=0.445\linewidth]{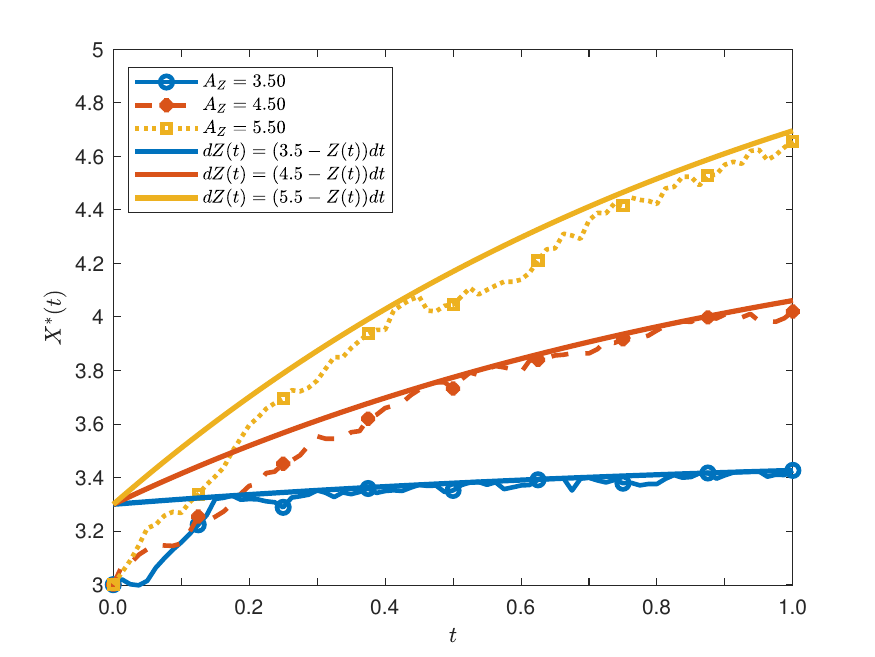}
		\label{Az-X}
	}
	\quad
	\subfigure[Abatement measure simulation with different values of the parameter $A_Z$.]{
		\includegraphics[width=0.445\linewidth]{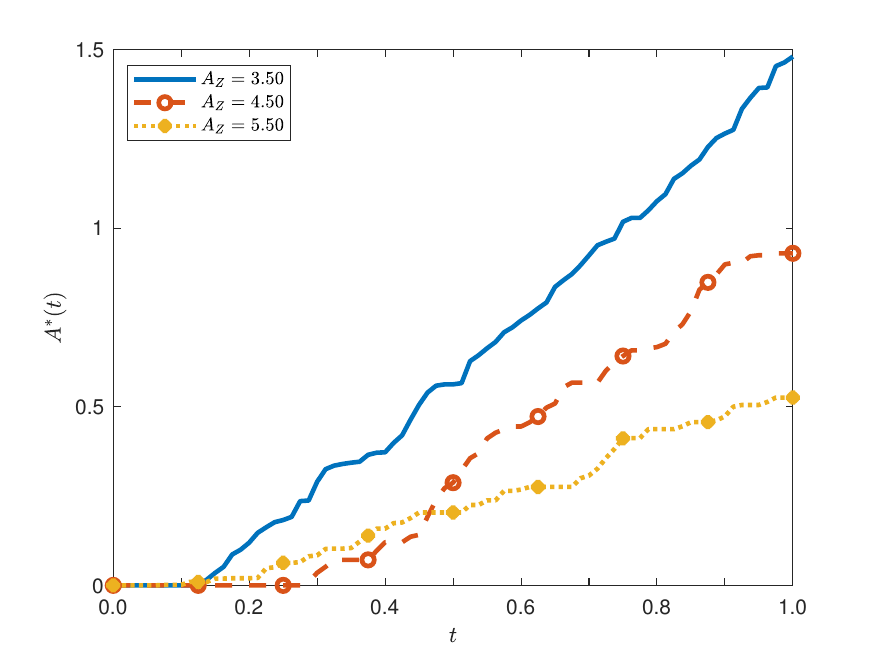}
		\label{Az-A}
	}\\
    \subfigure[Production strategy simulation with different values of the parameter $A_Z$.]{
		\includegraphics[width=0.445\linewidth]{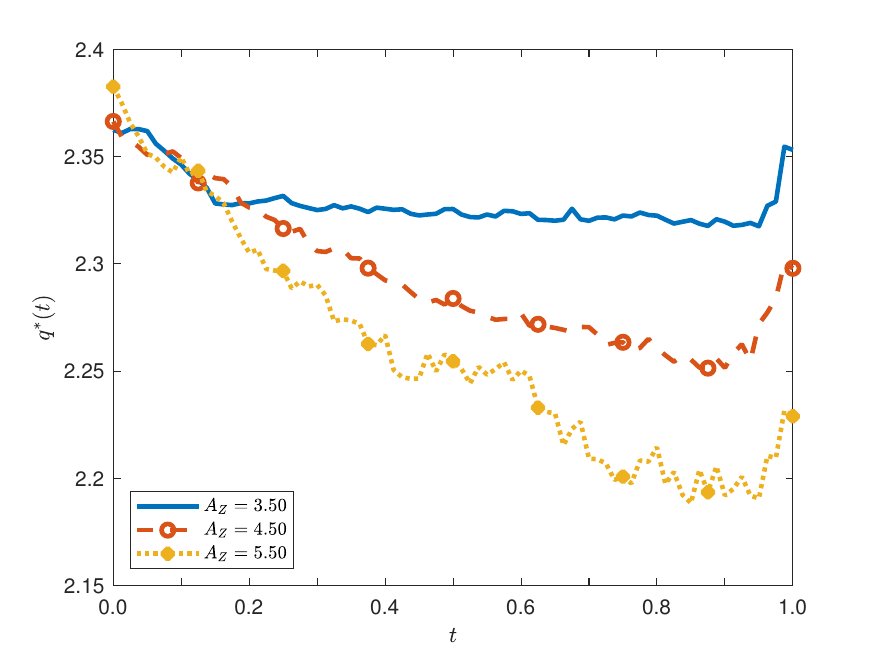}
		\label{Az-q}
	}
	\quad
	\subfigure[The leader's strategy simulation with different values of the parameter $A_Z$.]{
		\includegraphics[width=0.445\linewidth]{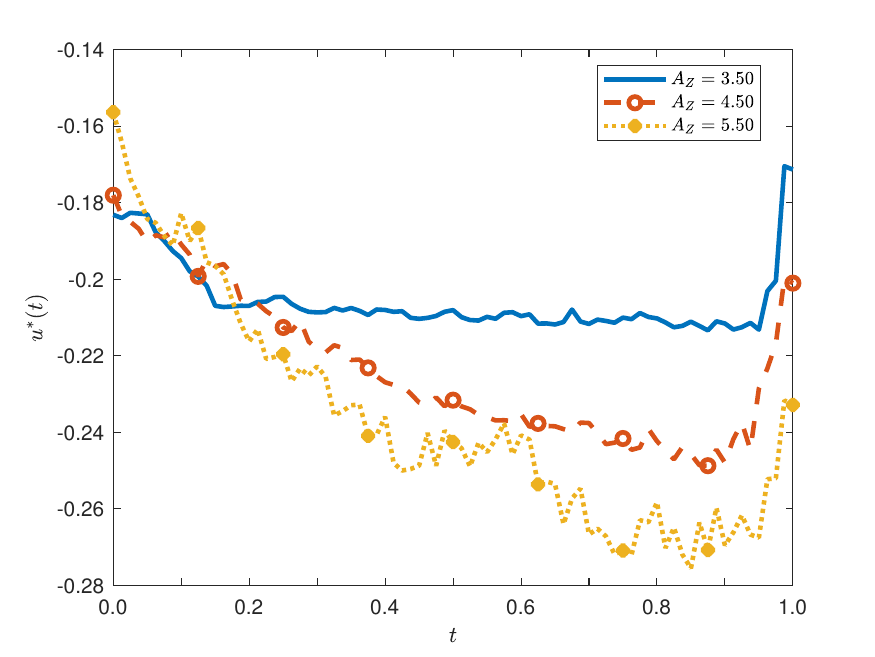}
		\label{Az-u}
	}
	\caption{Simulations w.r.t. different values of the parameter $A_Z$.}\label{fig:AZ}
\end{figure}

\quad For the simplicity of simulations, we simulate the homogeneous case by adopting the parameter values specified as production-emission conversion coefficient $\mu=1$,  volatility rate of average accumulated emission process $\sigma=0.1$, natural carbon absorption rate $\delta=0.1$,  unregulated baseline product price $a=10$, production cost parameter $c=2$, price adjustment cost parameter $c_{\rm P}=1$, abatement cost parameter $c_{\rm A}=0.5$,  environmental pollution penalty cost parameter $c_{\rm E}=0.4$, initial average cumulative carbon emission $x_0=3$, initial cumulative emission cap $z_0=3.3$.

\quad In what follows, we investigate how the steady‑state level coefficient \(A_Z\) and the adjustment speed coefficient \(B_Z\) of the emission cap process affect four key processes: the region's output \(q^*=(q^*(t))_{t\in[0,T]}\), the central regulator's price adjustment \(u^*=(u^*(t))_{t\in[0,T]}\), the abatement intensity \(A^*=(A^*(t))_{t\in[0,T]}\), and the average cumulative emission \(X^*=(X^*(t))_{t\in[0,T]}\). Numerical results show that changes in the two parameters induce opposite but internally consistent adjustments.

\begin{figure}[htbp]
	\centering
	\subfigure[Average emission process simulation with different values of the parameter $B_Z$.]{
		\includegraphics[width=0.445\linewidth]{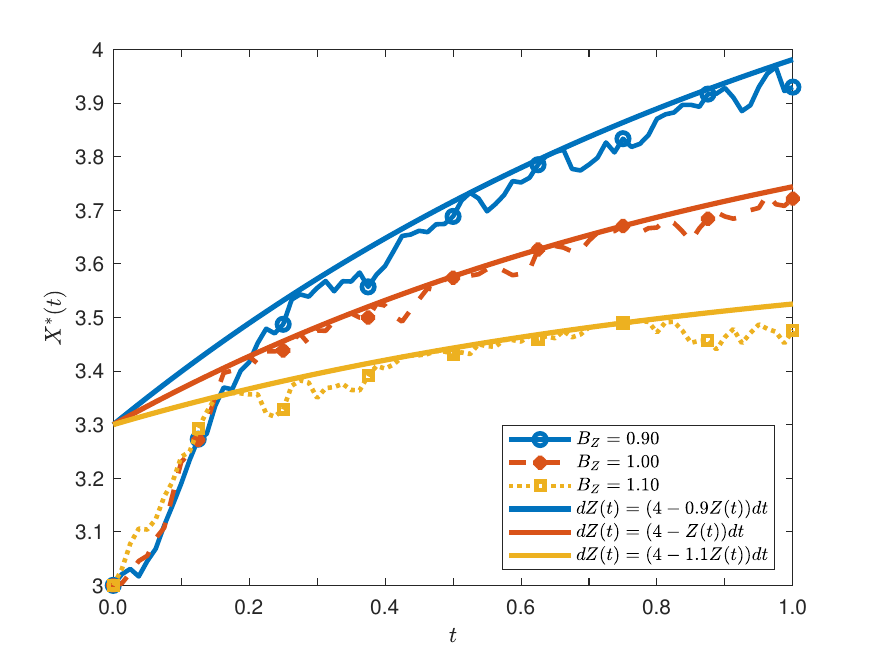}
		\label{Bz-X}
	}
	\quad
	\subfigure[Abatement measure simulation with different values of the parameter $B_Z$.]{
		\includegraphics[width=0.445\linewidth]{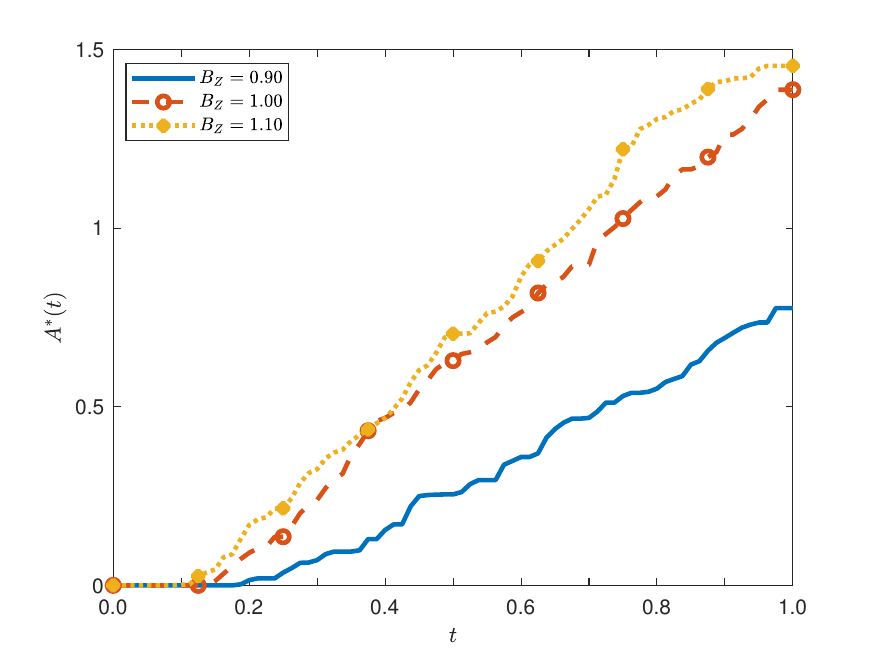}
		\label{Bz-A}
	}\\
    \subfigure[Production strategy simulation with different values of the parameter $B_Z$.]{
		\includegraphics[width=0.445\linewidth]{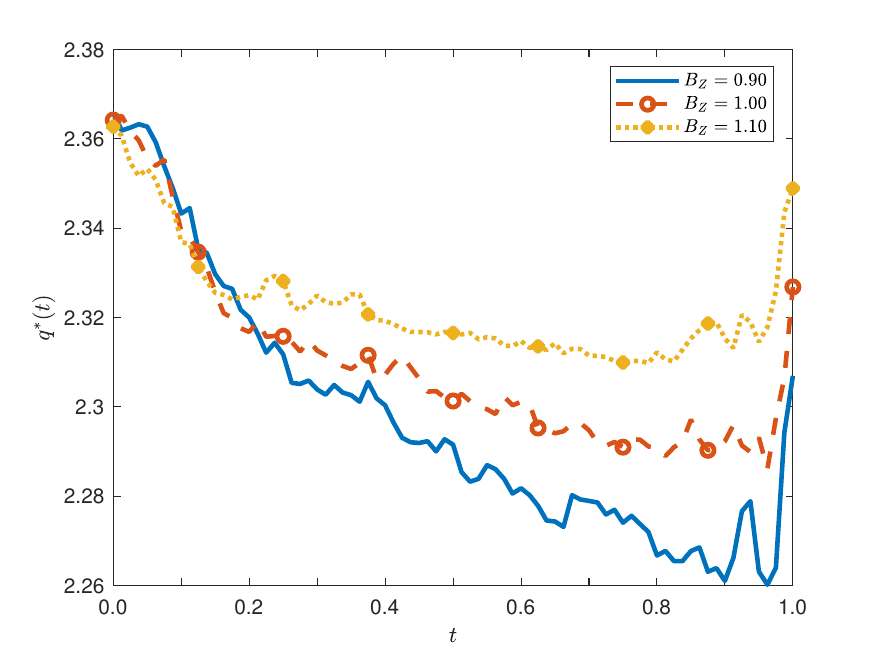}
		\label{Bz-q}
	}
	\quad
	\subfigure[The leader's strategy simulation with different values of the parameter $B_Z$.]{
		\includegraphics[width=0.445\linewidth]{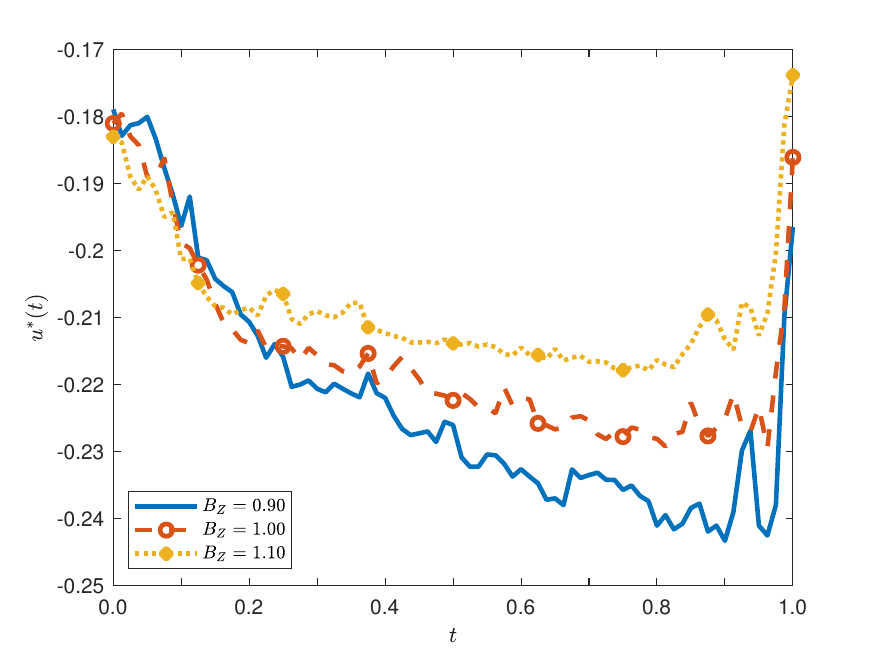}
		\label{Bz-u}
	}
	\caption{Simulations w.r.t. different values of the parameter $B_Z$.}\label{fig:BZ}
\end{figure}

\quad Figure~\ref{fig:AZ} examines the sensitivity of the average emission level $X^*$, leader's abatement strategy $A^*$ and the approximate Stackelberg equilibrium strategy pair $(q^*,u^*)$ with respect to variations in the steady-emission cap level coefficient $A_Z$, considering values of  $A_Z=3.5$ (blue line), $A_Z=4.5$ (red line) and $A_Z=5.5$ (yellow line), while $B_Z $ is chosen as $1$. For any given time $t$, 
fixing \(B_Z\) and increasing \(A_Z\) (i.e., the steady‑state equilibrium value of the emission cap, \(Z_{\infty}=A_Z/B_Z\), rises, and the overall emission constraint becomes looser) generally leads to a decreasing output \(q^*(t)\), an increasing absolute price adjustment \(|u^*(t)|\), and a decreasing abatement intensity \(A^*(t)\). In all cases, actual carbon emissions remain tight against the emission cap.

\quad Figure~\ref{fig:BZ} examines the sensitivity of the average emission level $X^{*}$, leader's abatement strategy $A^*$ and the approximate Stackelberg equilibrium strategy pair $(q^*,u^*)$ with respect to variations in the emission cap adjustment rate $B_Z$, considering values of $B_Z=0.9$ (blue line), $B_Z=1$ (red line) and $B_Z=1.1$ (yellow line), while $A_Z = 4$. For any given time $t$, 
fixing $A_Z$ and increasing $B_Z$ (i.e., the emission cap converges more quickly to a lower steady‑state equilibrium value, implying a tighter emission constraint and a more rapid policy adjustment pace) leads to an increasing output $q^*(t)$, a decreasing absolute price adjustment $|u^*(t)|$, and an increasing abatement intensity $A^*(t)$. Again, emissions consistently track the cap boundary.

\quad In both scenarios, the output process $q^*=(q^*(t))_{t\in[0,T]}$ exhibits a non‑monotonic pattern: it first declines over time and then rises near the terminal time $T$. This is a typical ``terminal effect" in finite‑horizon dynamic optimization. During the early and middle stages, the central regulator suppresses output via price adjustments to avoid exhausting future emission space and incurring high marginal abatement costs. As the horizon end approaches, the region anticipates weaker policy pressure and increases output to utilize remaining allowances and maximize terminal profits. The central regulator also slows price reductions in the terminal phase, further encouraging capacity release (as illustrated in Figures~\ref{Az-u} and \ref{Bz-u}).

\quad The economic interpretation is as follows. When the steady‑state level coefficient $A_Z$ increases, it indicates a higher equilibrium cap $Z_{\infty}=A_Z/B_Z$ and a looser constraint. Intuitively, a looser emission allowance should incentivize the region to expand production. However, the numerical results show the opposite: output decreases. This counter‑intuitive outcome arises from the central regulator's forward‑looking price adjustment. Facing a higher cap, the central regulator pre‑emptively cuts prices ($|u^*(t)|$ rises) to curb the region's incentive to increase output, thereby avoiding a rapid rise in emissions and costly high‑intensity abatement. Lower prices compress the region's profit margin, leading it to rationally reduce output. As a result, actual emissions fall, and the central regulator can lower abatement intensity $A^*(t)$, saving on increasing marginal abatement costs. Hence, a looser emission cap does not spur expansion but instead generates a ``contraction effect" through the central regulator's pre‑emptive price mechanism.

\quad In contrast, when the adjustment speed coefficient $B_Z$ increases (i.e., the emission cap tightens more quickly towards a lower level), the emission space becomes scarcer and highly predictable. For a fixed time $t$, facing the challenge of maintaining output under a tighter cap, the central regulator reduces the magnitude of price adjustments (i.e., $|u^*(t)|$ decreases w.r.t. $B_{Z}$), effectively raising the product price to improve the region's profit margin and encourage production. Simultaneously, the central regulator increases abatement intensity $A^*(t)$ to keep cumulative emissions within the lower cap, despite higher abatement costs. As a result, the region's output $q^*(t)$  rises, but the increase in abatement fully offsets the associated emissions, so total carbon emissions remain tightly controlled near the new, lower cap. Consequently, higher output is achieved with lower total emissions than under a loose constraint, albeit at the expense of increased abatement costs.


\quad These findings reveal key trade‑offs: raising the long‑run cap $A_Z$ may unexpectedly contract output due to the central regulator's pre‑emptive price response, whereas accelerating the tightening speed $B_Z$ if combined with appropriate price support – can achieve output growth and lower total emissions at the cost of higher abatement expenditure. Policy makers must carefully calibrate the emission cap design to balance environmental goals and economic efficiency.

\appendix 

\section{Proofs of Convergences \eqref{eq:momentestite0} and \eqref{eq:sup-i-bar} }\label{sec:proof}

In this Appendix, we provide the proof of the convergences \eqref{eq:momentestite0} and \eqref{eq:sup-i-bar}.
\vspace{-0.4em}
\begin{proof}[Proof of \eqref{eq:momentestite0}] 
Recall \eqref{eq:n-opt-X-F} and \eqref{eq:aux-MFG-followers}. We rewrite $X^{*,u,(n)}(t)=\widetilde{X}^{*,u,(n)}(t)-A^{*,u,(n)}(t)$ and $\bar{X}(t)=\widetilde{\bar{X}}(t)-\bar{A}(t)$ for $t\in[0,T]$. Here, the processes $\widetilde{X}^{*,u,(n)}=(\widetilde{X}^{*,u,(n)}(t))_{t\in[0,T]}$ and $\widetilde{\bar{X}}=(\widetilde{\bar{X}}(t))_{t\in[0,T]}$ satisfy respectively  the following SDEs, for $t\in[0,T]$,
\begin{align}
&d \widetilde{X}^{*,u,(n)}(t)=\bigg(\bigg(\mu \nu^{(n)}\left(\frac{1}{2 {c}}\right)u(t) -\delta\bigg) X^{*,u,(n)}(t)+\mu \nu^{(n)}\left(\frac{a}{2 {c}}\right)\bigg) d t\nonumber\\
&\hspace{6.5em}+\sigma X^{*,u,(n)}(t) d W(t),\label{eq:n-opt-X-tilde}\\
&d \widetilde{\bar{X}}(t)=\left(\left(\mu\nu\left(\frac{1}{2 {c}}\right) u(t) -\delta \right)\bar{X}(t)\right.\left.+\mu \nu\left(\frac{a}{2 {c}}\right)\right) d t+\sigma \bar{X}(t) d W(t)\label{eq:bar-X-tilde-dif}  
\end{align}
with the same initial conditions $\widetilde{X}^{*,u,(n)}(0)=x_0$ and $\widetilde{\bar{X}}(0)=x_0$. Then, we have 
\begin{align*}
&d \left(\widetilde{X}^{*,u,(n)}(t)-\widetilde{\bar{X}}(t)\right) =\left[\mu \left(\nu^{(n)}\left(\frac{a}{2 {c}}\right)-\nu \left(\frac{a}{2 {c}}\right) \right)+\mu \left(\nu^{(n)}\left(\frac{1}{2 {c}}\right)-\nu \left(\frac{1}{2 {c}}\right) \right)u(t) X^{*,u,(n)}(t)\right.\nonumber\\
&\qquad\left.+ \left(\mu\nu\left(\frac{1}{2 {c}}\right) u(t) -\delta \right) \left( {X}^{*,u,(n)}(t)- {\bar{X}}(t)\right)  \right] dt+\sigma \left( {X}^{*,u,(n)}(t)- {\bar{X}}(t)\right) d W(t)   
\end{align*} 
with $\widetilde{X}^{*,u,(n)}(0)-\widetilde{\bar{X}}(0)=0$. Hence, one has, for $t\in[0,T]$,
\begin{align}\label{diff:tilde-X-bar-star}
& \widetilde{X}^{*,u,(n)}(t)-\widetilde{\bar{X}}(t)=\mu \left(\nu^{(n)}\left(\frac{a}{2 {c}}\right)-\nu \left(\frac{a}{2 {c}}\right) \right)t+\int_{0}^{t} \Bigg[\mu \left(\nu^{(n)}\left(\frac{1}{2 {c}}\right)-\nu \left(\frac{1}{2 {c}}\right) \right) u(s) X^{*,u,(n)}(s)\nonumber\\
&~ + \left(\mu\nu\left(\frac{1}{ {c}}\right) u(s) -\delta \right) \left( {X}^{*,u,(n)}(s)- {\bar{X}}(s)\right)  \Bigg] ds+\sigma \int_{0}^{t}  \left( {X}^{*,u,(n)}(s)- {\bar{X}}(s)\right) d W(s).    
\end{align}
Set $Y^{(n)}(t):=\widetilde{X}^{*,u,(n)}(t)-\widetilde{\bar{X}}(t)$ for $t\in[0,T]$. By applying It\^{o}'s formula to $|Y^{(n)}(t)|^2$, together with Burkholder–Davis–Gundy (BDG) inequality, there exists a constant $C=C(T)>0$ such that
\begin{align*}
&\mathbb{E}\left[\sup_{0\le r\le t} \left|Y^{(n)}(r)\right|^2\right]\nonumber\\
&\leq\mathbb{E}\biggl[\int_0^t 2|Y^{(n)}(s)|\,\biggl| \mu \left(\nu^{(n)}\Bigl(\frac{a}{2 {c}}\Bigr)-\nu\Bigl(\frac{a}{2 {c}}\Bigr)\right) + \mu \left(\nu^{(n)}\Bigl(\frac{1}{2 {c}}\Bigr)-\nu\Bigl(\frac{1}{2 {c}}\Bigr)\right) u(s) X^{*,u,(n)}(s) \\
&\quad + \left(\mu\nu\Bigl(\frac{1}{2 {c}}\Bigr) u(s) - \delta\right) \bigl(X^{*,u,(n)}(s)-\bar{X}(s)\bigr) \biggr| ds\biggr] + \int_0^t \sigma^2 \mathbb{E}\bigl[\bigl(X^{*,u,(n)}(s)-\bar{X}(s)\bigr)^2\bigr] ds\\
&\quad + C\, \mathbb{E}\left[ \left( \int_0^t 4\sigma^2 Y^{(n)}(s)^2 \bigl(X^{*,u,(n)}(s)-\bar{X}(s)\bigr)^2 ds \right)^{1/2} \right].
\end{align*}
We also have from Cauchy's inequality that, for any $\epsilon>0$,
\begin{align*}
&\left( \int_0^t 4\sigma^2 |Y^{(n)}(s)|^2 |X^{*,u,(n)}(s)-\bar{X}(s)|^2 ds \right)^{1/2}\\
&\leq 2\sigma\left(\sup_{0\leq s\leq t}|Y^{(n)}(s)|^2\int_0^t|X^{*,u,(n)}(s)-\bar{X}(s)|^2\d s\right)^{\frac12}\\
&\leq \sigma\left(\xi\sup_{0\leq s\leq t}|Y^{(n)}(s)|^2+\frac{1}{\xi}\int_0^t|X^{*,u,(n)}(s)-\bar{X}(s)|^2\d s\right).
\end{align*}
Using the Lipschitz property of the Skorokhod mapping, it readily follows that, for any $t\in[0,T]$,
\begin{align}\label{eq;lipschitz_sk}
\left|X^{*,u,(n)}(t)-\bar{X}(t)\right|\leq \left|Y^{(n)}(t)\right|.
\end{align}
Inserting the above inequality into the above estimate of $\mathbb{E}[\sup_{0\le r\le t}|Y^{(n)}(r)|^2]$.  For any $u\in\mathbb{U}_0$, by applying Gronwall's inequality to conclude that, there exists a constant $C=C(T,m,M)$ such that
\begin{align*}
&\mathbb{E}\left[\sup_{t\in[0,T]}\left|Y^{(n)}(t)\right|^2\right]\leq C\E\left[\int_0^T\left|\mu \left(\nu^{(n)}\left(\frac{1}{2 {c}}\Bigr)-\nu\Bigl(\frac{1}{2 {c}}\right)\right) u(s) X^{*,u,(n)}(s)\right|^2ds\right],
\end{align*}
which converges to $0$ as $n\to\infty$ by dominated convergence theorem (DCT) since $u\in\mathbb{U}_0$ is uniformly $L^2$-bounded and $0<X^{*,u,(n)}(t)\leq Z(t)$ for all $t\in[0,T]$. The proof is therefore complete by using \eqref{eq;lipschitz_sk}.
\end{proof}

\vspace{-0.4cm}
\begin{proof}[Proof of \eqref{eq:sup-i-bar}]

Similar to the proof of \eqref{eq:momentestite0}, recall \eqref{eq:n-opt-X-F-i}, we  rewrite $X^{*,u,(n)}_{-i}(t)=\widetilde{X}^{*,u,(n)}_{-i}(t)-A^{*,u,(n)}_{-i}(t)$ for $t\in[0,T]$. Then, $\widetilde{X}_{-i}^{*,u,(n)}=(\widetilde{X}_{-i}^{*,u,(n)}(t))_{t\in[0,T]}$ satisfies that $\widetilde{X}^{*,u,(n)}_{-i}(0)=x_0$, and for $t\in(0,T]$,
\begin{align}\label{eq:n-opt-X-tilde-i} 
d\widetilde{X}^{*,u,(n)}_{-i}(t)=&\bigg[\frac{\mu}{n} q_i(t) +\mu\frac{n-1}{n} \nu^{(n)}_{-i}\left(\frac{1}{2 {c}}\right)u(t) X^{*,u,(n)}(t)+\mu \frac{n-1}{n}\nu^{(n)}_{-i}\left(\frac{a}{2 {c}}\right) -\delta X^{*,u,(n)}_{-i}(t)\bigg] d t\nonumber\\
&+\sigma X^{*,u,(n)}_{-i}(t) d W(t).
\end{align}
Recall that the process $\widetilde{\bar{X}}=(\widetilde{\bar{X}}(t))_{t\in[0,T]}$ is given by \eqref{eq:bar-X-tilde-dif}. Then, we have
\begin{align*}
&d\left(\widetilde{X}^{*,u,(n)}_{-i}(t)-\widetilde{\bar{X}}(t)\right)=\left(\frac{\mu}{n}\left(q_i(t)-\nu\left(\frac{a}{ {c}}\right)-\nu\left(\frac{1}{ {c}}\right) u(t) \bar{X}(t) \right)- \mu\frac{n-1}{n} \left(\nu^{(n)}_{-i}\left(\frac{a}{ {c}}\right)-\nu \left(\frac{a}{ {c}}\right) \right)\right.\\
&\left.-\delta \left( {X}_{-i}^{*,u,(n)}(t)- {\bar{X}}(t)\right)+\mu \frac{n-1}{n} \left(\nu^{(n)}_{-i}\left(\frac{1}{ {c}}\right)-\nu \left(\frac{1}{ {c}}\right) \right)u(t) \bar{X}(t)\right.\\
&\left.+ \mu\frac{n-1}{n}\nu_{-i}^{(n)}\left(\frac{1}{ {c}}\right) u(t)   \left( {X}^{*,u,(n)}(t)- {\bar{X}}(t)\right)\right) dt+\sigma \left( {X}^{*,u,(n)}_{-i}(t)- {\bar{X}}(t)\right) d W(t),  
\end{align*} 
where the initial data $\widetilde{X}^{*,u,(n)}_{-i}(0)-\widetilde{\bar{X}}(0)=0$. Hence, by following the similar argument of the proof of \eqref{eq:momentestite0} and utilizing the fact $\nu^{(n)}_{-i}(\frac{a}{2 {c}})\to\nu(\frac{a}{2 {c}})$ and $\nu^{(n)}_{-i}(\frac{1}{2 {c}})\to\nu(\frac{1}{2 {c}})$ as $n\to\infty$, we arrive at
\begin{align*}
\lim_{n\to\infty} \mathbb{E} \left[ \sup_{t\in[0,T]}\left| {X}^{*,u,(n)}_{-i}(t)- {\bar{X}}(t) \right|^2\right]=0.    
\end{align*}
Thus, the proof of the convergence is complete.
\end{proof}



\end{document}